\documentclass[11pt,           
draft,                         
english                       
]{article}

\usepackage[english,german,strings]{babel}     
\usepackage{amsmath}           
\usepackage{amssymb}
\usepackage{mathrsfs}
\usepackage{mathtools}
\usepackage{stmaryrd}
\usepackage{xkeyval}
\usepackage[utf8]{inputenc}    
\usepackage[T1]{fontenc}       
\usepackage{longtable}         
\usepackage{exscale}           
\usepackage[final]{graphicx}   
\usepackage[sort]{cite}        
\usepackage{array}             
\usepackage{fancyhdr}          
\usepackage[a4paper]{geometry} 
\usepackage[multiuser]{fixme}  
\usepackage{xspace}            
\usepackage{tikz}              
\usepackage{ifdraft}           
\usepackage[expansion=false    
           ]{microtype}        
\usepackage[nottoc]{tocbibind} 
\usepackage[backref=page,      
           final=true,         
           pdfpagelabels       
           ]{hyperref}         
\usepackage[amsmath,thmmarks,hyperref]{ntheorem}
\usepackage{graphicx}
\usepackage{enumitem}
\usepackage{aliascnt}

\usetikzlibrary{matrix,calc,fadings,decorations.pathreplacing,positioning,patterns}

\ifdraft{\synctex=1}{}

\DeclareSymbolFont{stmry}{U}{stmry}{m}{n}
\SetSymbolFont{stmry}{bold}{U}{stmry}{b}{n}

\numberwithin{equation}{section}

\allowdisplaybreaks

\let\originalleft\left
\let\originalright\right
\renewcommand{\left}{\mathopen{}\mathclose\bgroup\originalleft}
\renewcommand{\right}{\aftergroup\egroup\originalright}
\renewcommand{\cleardoublepage}{\clearpage\ifodd\c@page\else\vspace*{\fill}\thispagestyle{empty}\newpage\fi}

\numberwithin{equation}{section}

\allowdisplaybreaks

\theoremheaderfont{\normalfont\bfseries}
\theorembodyfont{\itshape}

\newaliascnt{corollary}{claim}
\newtheorem{corollary}[corollary]{Corollary}
\aliascntresetthe{corollary}
\newaliascnt{definition}{claim}
\newtheorem{definition}[definition]{Definition}
\aliascntresetthe{definition}
\newaliascnt{lemma}{claim}
\newtheorem{lemma}[lemma]{Lemma}
\aliascntresetthe{lemma}
\newaliascnt{proposition}{claim}
\newtheorem{proposition}[proposition]{Proposition}
\aliascntresetthe{proposition}
\newaliascnt{theorem}{claim}
\newtheorem{theorem}[theorem]{Theorem}
\aliascntresetthe{theorem}
\theorembodyfont{\rmfamily}
\newaliascnt{example}{claim}

\aliascntresetthe{example}
\newaliascnt{remark}{claim}
\newtheorem{remark}[remark]{Remark}
\aliascntresetthe{remark}
\theorembodyfont{\itshape}
\theoremnumbering{Roman}
\DeclareMathSymbol\chairxboxempty{\mathord}{AMSa}{"03}
\theoremheaderfont{\scshape}
\theorembodyfont{\normalfont}
\theoremstyle{nonumberplain}
\theoremseparator{:}
\theoremsymbol{\hbox{$\chairxboxempty$}}
\newtheorem{proof}{Proof}

\newenvironment{enumerate2}{%
  \begin{enumerate}[%
      topsep = 0.2em,
      partopsep = 0em,
      itemsep = 0.2em,
      parsep = 0.1em,
      label=\itshape\roman*.)%
      ]%
    }%
    {\end{enumerate}}

\DeclareMathAlphabet{\chairxmathbbm}{U}{bbm}{m}{n}
\SetMathAlphabet\chairxmathbbm{bold}{U}{bbm}{bx}{n}
\renewcommand{\mathbb}[1]{\chairxmathbbm{#1}}
\DeclareMathAlphabet{\mathscr}{U}{rsfso}{m}{n}
\DeclareSymbolFont{EulerScript}{U}{eus}{m}{n}
\SetSymbolFont{EulerScript}{bold}{U}{eus}{b}{n}
\DeclareSymbolFontAlphabet\mathcal{EulerScript}

\newcommand{\red}{\mathrm{red}} 
\newcommand{\std}{{\scriptscriptstyle{\mathrm{std}}}} 
\newcommand{\CE}{{\scriptscriptstyle{\mathrm{CE}}}}

\newcommand{\nice}{{\scriptscriptstyle{\mathrm{nice}}}}  
\newcommand{\BRST}{{\scriptscriptstyle{\mathrm{BRST}}}}  

\newcommand{\Cinfty}{\mathscr{C}^\infty}
\newcommand{\I}{\mathrm{i}}
\newcommand{\ring}[1]{\mathsf{#1}}
\newcommand{\group}[1]{\mathrm{#1}}
\newcommand{\liealg}[1]{\mathfrak{#1}}
\newcommand{\algebra}[1]{\mathscr{#1}}
\newcommand{\Anti}{\Lambda}
\newcommand{\HCE}{\mathrm{H}_\CE}
\newcommand{\prol}{\mathrm{prol}}
\newcommand{\cc}[1]{\overline{{#1}}}
\DeclarePairedDelimiter{\abs}{\lvert}{\rvert}
\newcommand{\id}{\operatorname{\mathrm{id}}}
\newcommand{\pr}{\operatorname{\mathrm{pr}}}
\newcommand{\argument}{\,\cdot\,}
\newcommand{\ad}{\operatorname{\mathrm{ad}}}
\newcommand{\Ad}{\operatorname{\mathrm{Ad}}}
\newcommand{\ins}{\operatorname{\mathrm{i}}}
\newcommand{\jns}{\operatorname{\mathrm{j}}}
\newcommand{\at}[2][\big]{#1\vert_{#2}}
\newcommand{\image}{\operatorname{\mathrm{im}}}
\newcommand{\prehilb}[1]{\mathcal{#1}}
\newcommand{\Bounded}{\mathscr{B}}
\newcommand{\End}{\operatorname{\mathrm{End}}}
\newcommand{\Clifford}{\operatorname{\mathrm{Cl}}}
\DeclarePairedDelimiter{\SP} {\langle}{\rangle}
\newcommand{\D}{\mathop{}\!\mathrm{d}}
\newcommand{\lefttriv}{\mathrm{left}}
\newcommand{\del}{\mathop{}\!\partial}
\newcommand{\basis}[1]{\mathit{#1}}
\newcommand{\Lie}{\mathscr{L}}
\newcommand{\tr}{\operatorname{\mathrm{tr}}}
\newcommand{\Span}[1][]{\operatorname{\mathrm{span}_{#1}}}

\DeclareMathOperator{\Gh} {\mathsf{Gh}}
\newcommand{\HBRST}{\mathrm{H}_\BRST}
\newcommand{\boldHBRST}{\boldsymbol{\mathrm{H}}_\BRST}
\newcommand{\boldHCE}{\boldsymbol{\mathrm{H}}_{\CE}}

\newcommand{\HBRSTtilde}{\widetilde{\mathrm{H}}_\BRST}
\newcommand{\boldHBRSTtilde}{\widetilde{\boldsymbol{\mathrm{H}}}_\BRST}
\newcommand{\BRSTalgebras} {\mathsf{BRST}{\textrm{-}\mathsf{Alg}}{}}
\newcommand{\iBRSTalg}   {\mathsf{iBRST}\textrm{-}^*{\mathsf{Alg}}{}}

\newcommand{\myemail}{\texttt{stefan.waldmann@mathematik.uni-wuerzburg.de}}
\newcommand{\myaddress}{Julius Maximilian University of Würzburg \\
     Department of Mathematics \\
     Chair of Mathematics X (Mathematical Physics) \\
     Emil-Fischer-Straße 31 \\
     97074 Würzburg \\
     Germany}

\newcommand{\AuthorOne}{Chiara Esposito}
\newcommand{\AuthorTwo}{Andreas Kraft}
\newcommand{\AuthorThree}{Stefan Waldmann}

\author{\AuthorOne\thanks{\AuthorEmailOne},
    \addtocounter{footnote}{2}
    \AuthorTwo\thanks{\AuthorEmailTwo}\\[0.5cm]
    \AuthorAddressOne \\[1cm]
    \AuthorThree\thanks{\AuthorEmailThree}\\[0.5cm]
    \AuthorAddressThree
}

\newcommand{\AuthorAddressOne}{Dipartimento di Matematica \\
     Universit\`a degli Studi di Salerno \\
     via Giovanni Paolo II, 132 \\
     84084 Fisciano (SA) \\
     Italy
}
\newcommand{\AuthorAddressThree}{\myaddress}

\newcommand{\AuthorEmailOne}{\texttt{chesposito@unisa.it}}
\newcommand{\AuthorEmailTwo}{\texttt{akraft@unisa.it}}
\newcommand{\AuthorEmailThree}{\myemail}

\title{BRST Reduction of Quantum Algebras with $^*$-Involutions}
\date{\today}

\begin{document}

\selectlanguage{english}

%
%

\maketitle

%
%

\begin{abstract}
    In this paper we investigate the compatibility of the BRST
    reduction procedure with the Hermiticity of star products.  First,
    we introduce the generalized notion of abstract BRST algebras with
    corresponding involutions.  In this setting we define adjoint BRST
    differentials and as a consequence one gets new BRST
    quotients. Passing to the quantum BRST setting we show that for
    compact Lie groups the new quantum BRST quotient and the quantum
    BRST cohomology are isomorphic in zero degree implying that
    reduction is compatible with Hermiticity.
\end{abstract}

\newpage

%
%

\tableofcontents
\newpage

%
%

\section{Introduction}
\label{sec:Introduction}

The aim of this work is to investigate the compatibility of the BRST
reduction in deformation quantization, as introduced in
\cite{bordemann.herbig.waldmann:2000a}, with the Hermiticity of star
products.
Deformation quantization as introduced in \cite{bayen.et.al:1978a} by
Bayen, Flato, Fronsdal, Lichnerowicz and Sternheimer relies on the
idea that the quantization of a symplectic or Poisson manifold $M$
representing the phase space of a classical mechanical system is
described by a formal deformation of the commutative algebra of smooth
complex-valued functions $\Cinfty(M)$.  Explicitly, one defines a
\emph{star product} $\star$ on $M$ being a
$\mathbb{C}[[\lambda]]$-bilinear associative product on
$\Cinfty(M)[[\lambda]]$ of the form
\begin{equation}
 	f\star g
	=
	f\cdot g + \sum_{r=1}^\infty \lambda^r C_r(f,g)
\end{equation}
for any $f,g\in\Cinfty(M)[[\lambda]]$, where $C_1(f,g) - C_1(g,f) = \I
\{f,g\}$ and where all the terms $C_r$ are bidifferential operators
vanishing on constants.  Here the formal parameter $\lambda$ is
supposed to be real.
Thus the quantum observables are described by the non commutative
algebra $(\Cinfty(M)[[\lambda]], \star)$.  In order to get a *-algebra
structure on the quantum observables we need to consider a
*-involution for the star product.  One calls the star product
\emph{Hermitian} if the complex conjugation is an involution, i.e. if
$\cc{f \star g} = \cc{g} \star \cc{f}$ for all $f,g\in
\Cinfty(M)[[\lambda]]$. The existence and classification of general
star products on Poisson manifolds has been provided by Kontsevich's
famous formality theorem \cite{kontsevich:2003a} and the existence of
Hermitian star products on symplectic manifolds was shown in
\cite{neumaier:2001a,neumaier:2002a}.

At the classical level it is possible to perform, under fairly general
conditions, the phase space reduction which constructs from the
original phase space $M$ one of a smaller dimension denoted by
$M_\red$, see e.g. \cite{marsden.weinstein:1974a}.
More precisely, suppose that a Lie group $\group{G}$ acts by symplectomorphisms
resp. Poisson diffeomorphisms and that it allows an $\Ad^*$-equivariant
momentum map $J\colon M\longrightarrow \liealg{g}^*$ with $0\in \liealg{g}^*$
as value and regular value, where $\liealg{g}$ denotes the Lie algebra of
$\group{G}$. Then $C = J^{-1}(\{0\})$ is a closed embedded submanifold of
$M$, called regular constraint surface. If the action is in addition
proper and free, then the reduced manifold $M_\red$ given by the orbit
space $C / \group{G}$ is again a symplectic resp. Poisson manifold.

In the setting of deformation quantization a quantum reduction scheme
has been introduced in \cite{bordemann.herbig.waldmann:2000a},
see also \cite{dippell.esposito.waldmann:2019a} for a more
categorical approach to reduction in both the quantum and classical setting.
One of the crucial ingredients is the notion of quantum momentum
maps \cite{xu:1998a}. Given a star product on $M$, a quantum momentum map is
a map $\boldsymbol{J} = J + \sum_{r=1}^\infty \lambda^r \boldsymbol{J}_r
\colon M \longrightarrow \liealg{g}^*[[\lambda]]$ into the formal
series of smooth functions on $M$ such that
\begin{equation}
  \boldsymbol{J}(\xi) \star f - f \star \boldsymbol{J}(\xi)
  =
  \I \lambda \{J(\xi),f\}
  \quad
  \text{and}
  \quad
  \boldsymbol{J}(\xi) \star \boldsymbol{J}(\eta)
  -\boldsymbol{J}(\eta) \star \boldsymbol{J}(\xi)
  =
  \I \lambda \boldsymbol{J}([\xi,\eta])
\end{equation}
for all $\xi,\eta\in\liealg{g}$ and $f\in \Cinfty(M)[[\lambda]]$. Here
$\boldsymbol{J}(\xi) =\SP{\boldsymbol{J},\xi}$ denotes the pointwise
dual pairing, see
\cite{gutt.rawnsley:2003a,mueller-bahns.neumaier:2004a}.  The map
$\boldsymbol{J}$ is called \emph{quantum momentum map} and the pairs
$(\star,\boldsymbol{J})$ are called \emph{equivariant star products},
see also \cite{reichert.waldmann:2016a,reichert:2017a,reichert:2017b}
for a classification in the symplectic setting.

The BRST approach provides then a tool to construct a reduced star
product $\star_\red$ on $M_\red$ that is induced by the equivariant
star product $(\star,\boldsymbol{J})$ on $M$ and thus implies that the
deformation quantization is compatible with the classical phase space
reduction.  Here the abbreviation BRST stands for the particle
physicists Becchi, Rouet, Stora \cite{becchi.rouet.stora:1976a} and
Tyutin \cite{tyutin:2008a} who investigated gauge invariances by
introducing new variables, the ``ghosts'' and ``antighosts''. see
also \cite{henneaux.teitelboim:1992a} for further applications in
physics.  Kostant and Sternberg \cite{kostant.sternberg:1987a}
transferred this idea to the setting of symplectic resp. Poisson geometry,
introducing the classical BRST algebra $\mathcal{A}^{\bullet,\bullet}
=\Anti^\bullet \liealg{g}^* \otimes\Anti^\bullet \liealg{g} \otimes
\Cinfty(M)$ with ghost number grading $\mathcal{A}^{(n)} =
\bigoplus_{n=k-\ell}\mathcal{A}^{k,\ell}$ and a corresponding super Poisson
structure induced by the natural pairing of $\liealg{g}^*$ and
$\liealg{g}$. The two characteristic features of the classical BRST
algebra are the classical BRST operator $D \colon
\mathcal{A}^{(\bullet)} \longrightarrow \mathcal{A}^{(\bullet+1)}$,
satisfying $D^2=0$, and the ghost number derivation $\Gh$ inducing the
ghost number grading.  With these notions it was shown that one has
the following isomorphism of Poisson algebras
\begin{equation}
  \label{eq:ZeroBRSTCohomologyCongMred}
  \HBRST^{(0)}(\mathcal{A})
  \cong
  \Cinfty(M_\red),
\end{equation}
where the classical BRST cohomology $\HBRST^{(\bullet)}(\mathcal{A})$
is the cohomology of $(\mathcal{A}^{(\bullet)},\{\argument,\argument\},D)$,
see also \cite{forger.kellendonk:1992a}.  As mentioned above,
Bordemann, Herbig and Waldmann \cite{bordemann.herbig.waldmann:2000a}
transferred this result to the setting of deformation quantization and
constructed the standard ordered quantum BRST algebra
$(\mathcal{A}^{(\bullet)}[[\lambda]],\star_\std,\boldsymbol{D}_\std)$
as formal deformation of the classical BRST algebra. In particular,
they proved the quantum analogue of
\eqref{eq:ZeroBRSTCohomologyCongMred}, namely
\begin{equation}
  \mathcal{A}_\red
  =
  \boldHBRST^{(0)}(\mathcal{A}[[\lambda]])
  \cong
  \Cinfty(M_\red)[[\lambda]].
\end{equation}
Here $\boldHBRST^{(\bullet)}(\mathcal{A}[[\lambda]])$ denotes the
cohomology of the quantized BRST algebra, the so-called quantum BRST
cohomology, and the ghost number zero part $\mathcal{A}_\red$ is
called reduced quantum BRST algebra.
The above construction induces a star product $\star_\red$ on the
reduced manifold, but if the star product on $M$ is Hermitian, the
construction does not yield an involution for it.
The main problem here is that in general homological algebra
is not compatible with involutions and positive definite inner products. Therefore, the new
question addressed in this work is whether one can modify the BRST
reduction in such a way that it gives in addition an induced
*-involution for the reduced star product. Note that there is also a
different way to construct involutions for $\star_\red$ via
*-representations, see \cite{bordemann:2005a,gutt.waldmann:2010a}.
A more general reduction scheme can also be found in
\cite{cattaneo.felder:2007a}.

To this end we introduce a notion of abstract BRST algebras and
investigate various concepts of involutions that are compatible with
the gradings.  We show that graded *-involutions with imaginary ghost
operator are the best suited involutions as they guarantee the
existence of non-trivial *-representations on pre-Hilbert spaces,
which is necessary from the physical point of view to encode for
example the superposition principle.  Applying these abstract results
to the setting of deformation quantization, we construct such an
involution for the quantum BRST algebra $\mathcal{A}[[\lambda]]$ by
means of a positive definite inner product on the Lie algebra as
additional information. In this case, we prove that the so constructed
*-algebra has sufficiently many positive linear functionals in
the sense of \cite{bursztyn.waldmann:2000a,bursztyn.waldmann:2001a,
bursztyn.waldmann:2005a,bursztyn.waldmann:2005b},
guaranteeing a non-trivial *-representation theory via GNS representations,
see also \cite{bordemann.waldmann:1998a}. Finally, we
introduce the adjoint quantum BRST operator $\boldsymbol{D}_\std^*$
and the quantum BRST quotient
\begin{equation}
  \boldHBRSTtilde^{(\bullet)} (\mathcal{A}[[\lambda]])
  =
  \frac{\ker \boldsymbol{D}_\std\cap \ker \boldsymbol{D}_\std^*}
       {\image \boldsymbol{D}_\std \cap \image \boldsymbol{D}_\std^*}.
\end{equation}
We show for compact Lie groups that its zero-th order is isomorphic
to the reduced BRST algebra, i.e.
\begin{equation}
  \boldHBRSTtilde^{(0)} (\mathcal{A}[[\lambda]])
  \cong
  \mathcal{A}_\red \cong \Cinfty(M_\red)[[\lambda]].
\end{equation}
The crucial ingredient in the proof is a
$\Cinfty(C)^\group{G}[[\lambda]]$-valued inner product, similarly to
algebra-valued inner products on Hilbert-modules \cite{lance:1995a}, but
over $\mathbb{C}[[\lambda]]$ as in \cite{gutt.waldmann:2010a}.
In particular, this isomorphism induces the complex conjugation as
involution for $\star_\red$, hence the BRST reduction of Hermitian
star products yields in this setting Hermitian reduced star products.
In other words, we show that
$\boldHBRST^{(\bullet)}(\mathcal{A}[[\lambda]])$ and
$\boldHBRSTtilde^{(\bullet)}(\mathcal{A}[[\lambda]])$ are isomorphic
in ghost number zero if the Lie group acting on $M$ is compact, which
provides a large class of examples for the physically relevant
invariants.

The paper is organized as follows: In
Section~\ref{section:Preliminaries} we recall the basics concerning
the classical BRST algebra and its counterpart in deformation
quantization. Then we introduce in
Section~\ref{section:AbstractBRSTalgebras} the notion of abstract BRST
algebras and look for compatible involutions and their *-representation
theory. Having found a suitable
concept of involutions we apply this idea in
Section~\ref{section:InvolutionsforQuantumBRSTAlgebra} at first to the
Grassmann part and then finally to the quantum BRST algebra. The
results of this paper are partially based on the master thesis
\cite{kraft:2018a}.

%
%
\section{Preliminaries}
\label{section:Preliminaries}

%
%
\subsection{The Classical BRST Complex and Cohomology}
\label{subsection:ClassicalBRSTComplexandCohomology}

In this section we recall the description of the classical
Marsden-Weinstein reduction via the classical BRST cohomology in order
to establish the notation. We refer to
\cite{bordemann.herbig.waldmann:2000a,forger.kellendonk:1992a,
  kostant.sternberg:1987a}.

Let us consider a Hamiltonian $\group{G}$-space
$(M,\group{G},J)$ consisting in a symplectic or Poisson manifold
$(M,\omega)$ resp. $(M,\pi)$ and a Hamiltonian action
$\Phi \colon \group{G} \times M \longrightarrow M$ with momentum map $J$.
It is well-known that the
quotient $M_\red = C / \group{G}$, where $C= J^{-1}(\{0\})$ with $0$
being a regular value of $J$, inherits a symplectic resp. Poisson structure
from $M$ if the action is free and proper, see
\cite{marsden.weinstein:1974a}. In addition, we can identify
$\Cinfty(M_\red)$ with $\Cinfty(C)^\group{G}$. From now on we call
$(M,\group{G},J,C)$ \emph{Hamiltonian $\group{G}$-space with
  regular constraint surface} and we denote by $\iota \colon C
\rightarrow M$ the canonical embedding and by $\mathcal{I}_C = \ker
\iota^*$ the \emph{vanishing ideal} of $C$.  Using a tubular
neighbourhood one can construct a \emph{prolongation map}
\begin{equation}
  \prol \colon
  \Cinfty(C)
  \longrightarrow
  \Cinfty(M)
\end{equation}
with $\iota^* \prol = \id\at{\Cinfty(C)}$, see
\cite[Lemma~2]{bordemann.herbig.waldmann:2000a}.  This yields in
particular $\Cinfty(C) \cong \Cinfty(M) /\mathcal{I}_C$.
Note that if the action is proper on $M$, then the prolongation map
can even be chosen to be $\group{G}$-equivariant.

We aim to give another description of the Poisson algebra
$\Cinfty(M_\red)$.  Let us consider the
$\mathbb{Z}\times\mathbb{Z}$-graded vector space
\begin{equation}
  \mathcal{A}^{\bullet,\bullet}
  =
  \Anti^\bullet \liealg{g}^* \otimes
  \Anti^\bullet\liealg{g} \otimes
  \Cinfty(M),
\end{equation}
where the gradings are also called \emph{ghost} and \emph{antighost
  degree}.  Then $\mathcal{A}$ carries a natural $\mathbb{Z}_2$-graded
vector space structure $\mathcal{A} = \Anti^{\text{even}}(\liealg{g}^*
\oplus \liealg{g})\otimes \Cinfty(M) \oplus
\Anti^{\text{odd}}(\liealg{g}^*\oplus \liealg{g})\otimes \Cinfty(M)$.
The $\mathbb{Z}$-grading
\begin{equation}
  \mathcal{A}^{(n)}
  =
  \bigoplus_{n= k-\ell} \mathcal{A}^{k,\ell},
\end{equation}
is called the \emph{ghost number} or \emph{total degree}. In
particular, the ghost number grading and the $\mathbb{Z}\times
\mathbb{Z}$-grading induce the same $\mathbb{Z}_2$-grading, so the
notions of super derivations with respect to the
$\mathbb{Z}_2$-grading and of graded derivations with respect to the
ghost number grading coincide.  With the $\wedge$-product of forms
$(\alpha \otimes \xi)\wedge (\beta \otimes \eta) = (-1)^{k\ell} (\alpha
\wedge \beta) \otimes (\xi \wedge \eta)$ for $\alpha \in \Anti^\bullet
\liealg{g}^*, \beta\in\Anti^k\liealg{g}^*, \xi \in\Anti^\ell\liealg{g}$
and $\eta\in \Anti^\bullet\liealg{g}$ and the pointwise product of
functions, $\mathcal{A}$ becomes an associative, super-commutative
algebra that is graded with respect to all the above mentioned
degrees. The element $ 1 \otimes 1\otimes 1$ is a unit and one has the
following differentials:
\begin{itemize}
\item The vertical differential is the Chevalley-Eilenberg
  differential
  \begin{equation}
    \label{eq:CEonBRST}
    \delta \colon
    \Anti^\bullet \liealg{g}^* \otimes
    \Anti^\bullet\liealg{g} \otimes
    \Cinfty(M)
    \longrightarrow
    \Anti^{\bullet+1} \liealg{g}^* \otimes
    \Anti^\bullet\liealg{g} \otimes
    \Cinfty(M),
  \end{equation}
  where the representation of $\liealg{g}$ on $\Anti^\bullet
  \liealg{g}\otimes \Cinfty(M)$ is defined by
  \begin{equation}
    \label{eq:ActionofgonKoszulComplex}
    \liealg{g}\ni \xi
    \mapsto
    \rho(\xi)
    =
    \ad(\xi) \otimes \id + \id \otimes \{J(\xi),\argument\} \in
	      \End(\Anti^\bullet \liealg{g}\otimes \Cinfty(M)).
  \end{equation}
  The corresponding cohomology is denoted by
  $\HCE^\bullet(\mathcal{A})$.
\item The horizontal differential
  $\del\colon\mathcal{A}^{\bullet,\bullet} \longrightarrow
  \mathcal{A}^{\bullet,\bullet-1}$ is the extended Koszul
  differential, explicitly given by $\del(\alpha\otimes x \otimes f) =
  (-1)^k \alpha \otimes \ins(J)(x\otimes f)$ for all $\alpha \in
  \Anti^k\liealg{g}^*, x\in\Anti^\bullet\liealg{g}$ and $f\in
  \Cinfty(M)$.
	Here $\ins(J)$ means the left insertion of $J$, i.e. the
	standard interior product.
\end{itemize}
One can show that $(\mathcal{A}^{\bullet,\bullet}, \del, \delta)$ is a double
complex and that the total differential
\begin{equation}
  \label{eq:TotalDifferential}
  D
  =
  \delta + 2 \del \colon
  \mathcal{A}^{(\bullet)}
  \longrightarrow
  \mathcal{A}^{(\bullet+1)}
\end{equation}
is a well-defined coboundary operator on the total complex
$\mathcal{A}^{(\bullet)}$, the so-called \emph{classical BRST
  operator}, see \cite[Section 4]{bordemann.herbig.waldmann:2000a}.
Note that the factor $2$ in front of the Koszul differential in
\eqref{eq:TotalDifferential} is just a convention and that the ghost
and antighost degrees are not respected by $D$, but the total degree
is.

It turns out that $\mathcal{A}^{(\bullet)}$ has also a natural super
Poisson stucture induced by the natural pairing of $\liealg{g}$ and
$\liealg{g}^*$.  Concerning the compatibility of this super Poisson
structure with the grading and the BRST operator one finds the
following properties:

\begin{itemize}
\item Let $2\gamma$ be the identity endomorphism of $\liealg{g}$,
  regarded as an element $\gamma = \frac{1}{2}\basis{e}^a \wedge
  \basis{e}_a \in \mathcal{A}^{1,1}$ in terms of a basis $\basis{e}_1,\dots,
	\basis{e}_n$ of $\liealg{g}$ with dual basis $\basis{e}^1,\dots, \basis{e}^n$.
	Then the ghost number grading of $\mathcal{A}^{(\bullet)}$ is induced by the
  \emph{ghost number derivation} $\Gh = \{\gamma,\argument\}$, i.e.  $\phi
  \in \mathcal{A}^{(k)}$ if and only if $\Gh\phi = k \phi$.  The
  element $\gamma\in \mathcal{A}^{(0)}$ is called \emph{ghost charge}.
\item The total differential $D$ fulfils $D=\{\Theta,\argument\}$ with
  $\Theta = \Omega + J$ and $\Omega = -\frac{1}{2} [\argument,\argument] =
  -\frac{1}{4}f^i_{jk} \basis{e}^j\wedge\basis{e}^k\wedge
  \basis{e}_i$, where $f^i_{jk}$ are the structure constants of $\liealg{g}$.
	In particular, the classical BRST operator is an
  inner Poisson derivation of degree $1$ and the odd element $\Theta
  \in \mathcal{A}^{(1)}$ is called \emph{classical BRST charge}.
\end{itemize}
Summarizing, one calls the differential $\mathbb{Z}$-graded super
Poisson algebra $(\mathcal{A}^{(\bullet)},D,\{\argument,\argument\})$
\emph{classical BRST algebra}, and the corresponding cohomology
$\HBRST^{(\bullet)}(\mathcal{A}) = \ker D/\image D$ \emph{classical
  BRST cohomology}. Since the classical BRST operator is an inner
Poisson derivation, it immediately follows that
$\HBRST^{(\bullet)}(\mathcal{A})$ inherits a $\mathbb{Z}$-graded super
Poisson structure from the classical BRST algebra. Moreover, $[1] \in
\HBRST^{(\bullet)}(\mathcal{A})$ is a unit with respect to the
$\wedge$-product, see \cite[Lemma~9]{bordemann.herbig.waldmann:2000a}.
It has been proved that in ghost number zero one has the isomorphism
\begin{equation}
  \label{eq:ClassicalIsoBRSTCE}
  \HBRST^{(0)}(\mathcal{A})
  \cong
  \HCE^0(\liealg{g}, \Cinfty(C))
  \cong
  \Cinfty(M_\red)
\end{equation}
of Poisson algebras, inducing a Poisson structure on the reduced
manifold, see \cite[Prop. 10]{bordemann.herbig.waldmann:2000a}.

%
%
\subsection{The Quantum BRST Complex and Cohomology}
\label{section:QuantumBRSTComplexandCohomology}

In a similar fashion, now one can perform all the above constructions
in the framework of deformation quantization, where we follow again
\cite{bordemann.herbig.waldmann:2000a}.  The underlying vector space
of the quantum BRST algebra are the formal power series
$\mathcal{A}^{(\bullet)}[[\lambda]]$ with values in the classical BRST
algebra, inheriting all the gradings of $\mathcal{A}$.

Let $(M,\group{G},J)$ be a Hamiltonian $\group{G}$-space with
star product $\star$. A \emph{quantum momentum map} is a formal series
$\boldsymbol{J} = \sum_{r=0}^\infty \lambda^r \boldsymbol{J}_r\colon M
\longrightarrow \liealg{g}^*[[\lambda]]$ of smooth functions
$\boldsymbol{J}_r\colon M\longrightarrow \liealg{g}^*$ such that
$\boldsymbol{J}_0 = J$ and such that $\boldsymbol{J}$ satisfies
\begin{equation}
  \label{eq:DefiEquivariantStarProduct}
  \boldsymbol{J}(\xi) \star \boldsymbol{J}(\eta)
  - \boldsymbol{J}(\eta) \star \boldsymbol{J}(\xi)
  = \I\lambda \textbf{J}([\xi,\eta])
  \quad
  \text{and}
  \quad
  \boldsymbol{J}(\xi)\star f
  -  f \star  \boldsymbol{J}(\xi)
  =
  \I \lambda \{J(\xi),f\}
\end{equation}
for all $\xi,\eta \in \liealg{g}$ and $f\in \Cinfty(M)[[\lambda]]$.
The pair $(\star,\boldsymbol{J})$ is called \emph{equivariant star
  product}.  The first property in
\eqref{eq:DefiEquivariantStarProduct} is also called \emph{quantum
  covariance} and ensures that the quantum momentum map is a morphism
of Lie algebras. Moreover, it implies that a Lie algebra
representation $\boldsymbol{\rho}_M$ is given by
\begin{equation}
  \label{eq:gRepQuantumAlgM}
  \boldsymbol{\rho}_M(\xi)
  =
  \frac{1}{\I \lambda} \ad(\boldsymbol{J}(\xi))
\end{equation}
for $\xi\in\liealg{g}$.  The second property implies that the star
product is also $\group{G}$-\emph{invariant}, i.e. satisfies
$\Phi_g^*(f\star h)=\Phi_g^*(f)\star \Phi_g^*(h)$ for all $g \in
\group{G}, f,h \in \Cinfty(M)[[\lambda]]$, and that $\boldsymbol{\rho}_M$
coincides with the classical action $\rho_M(\xi) = - \Lie_{\xi_M}$.
The quadruple $(M,\star,\group{G},\boldsymbol{J})$ is called
\emph{Hamiltonian quantum $\group{G}$-space} if
$(M,\group{G},J)$ is a Hamiltonian $\group{G}$-space with
equivariant star product $(\star,\boldsymbol{J})$.  Similarly, we call
$(M,\star,\group{G},\boldsymbol{J},C)$ a \emph{Hamiltonian quantum
  $\group{G}$-space with regular constraint surface} if
$(M,\group{G},J,C)$ is a Hamiltonian $\group{G}$-space with
regular constraint surface and equivariant star product
$(\star,\boldsymbol{J})$.  Recall that in the symplectic case one
can construct for every proper and strongly
Hamiltonian group action an equivariant star product
$(\star,J)$, i.e. with $\boldsymbol{J}=J$, see
\cite[Sect.~5.8]{fedosov:1996a}. Such star products are also called
\emph{strongly invariant}.

Therefore, we assume from now on that
$(M,\star,\group{G},\boldsymbol{J},C)$ is a Hamiltonian quantum
$\group{G}$-space with regular constraint surface.  A quantized
version of the Grassmann part $\Anti^\bullet(\liealg{g}^* \oplus
\liealg{g})$ can be constructed in the following way, see
\cite[Sect.~5]{bordemann.herbig.waldmann:2000a}.

Let $\mu$ denote the $\wedge$-product. The \emph{standard ordered star
product} $\circ_\std$ for $ \Anti^\bullet (
\liealg{g}_\mathbb{C}^*\oplus \liealg{g}_\mathbb{C})[[\lambda]]$
is defined by
\begin{equation}
  \label{eq:DefiStarStdGrassmann}
  a \circ_\std b
  =
  \mu \circ e^{2\I\lambda P^*} a \otimes b
\end{equation}
with $P^* = \jns(\basis{e}^k) \otimes \ins(\basis{e}_k)$, where $\jns$
denotes the right insertion and $\ins$ the left insertion.
It is a formal deformation of the $\wedge$-product: it is a
$\mathbb{C}[[\lambda]]$-bilinear, associative map such that for
all homogeneous $a,b\in\Anti^\bullet (
\liealg{g}_\mathbb{C}^*\oplus \liealg{g}_\mathbb{C})[[\lambda]]$
one has $1\circ_\std a = a = a\circ_\std 1$ and
\begin{equation}
  a \circ_\std b
  =
  a \wedge b + \sum_{r=1}^\infty \lambda^r C_r(a,b)
\end{equation}
with $C_1(a,b)-(-1)^{\abs{a}\abs{b}}C_1(b,a) = \I \{a,b\}$.

Tensoring the standard ordered star product on the Grassmann part with
the equivariant star product $(\star,\boldsymbol{J})$ for the
functions, we obtain an associative product $\star_\std$ for
$\mathcal{A}[[\lambda]]$. Explicitly, we have for $a, b\in
\Anti^\bullet (\liealg{g}_\mathbb{C}^*\oplus
\liealg{g}_\mathbb{C})[[\lambda]]$ and $f,g\in \Cinfty(M)[[\lambda]]$
\begin{equation}
  \label{eq:StarProductQuantumBRSTAlgebra}
  (a \otimes f) \star_\std (b \otimes g)
  =
  (a \circ_\std b) \otimes (f \star g).
\end{equation}
In analogy to the classical
case one defines the \emph{standard ordered quantum BRST charge} by
\begin{equation}
  \Theta_\std
  =
  \Omega + \boldsymbol{J} + \I \lambda \chi.
\end{equation}
Here $\chi\in \liealg{g}^* \subset \mathcal{A}^{(1)}[[\lambda]]$ is
defined by $\chi(\xi)=\frac{1}{2}\tr(\ad(\xi))$ for all
$\xi\in\liealg{g}$, whence $\Theta_\std$ coincides in the zero-th
order of $\lambda$ with the classical $\Theta$. One can compute
$\Theta_\std \star_\std \Theta_\std = 0$. Consequently, the
\emph{standard ordered quantum BRST operator} is given by
\begin{equation}
  \boldsymbol{D}_\std
  =
  \frac{1}{\I\lambda} \ad_\std(\Theta_\std),
\end{equation}
where $\ad_\std$ denotes the taking of the super commutator with
respect to the standard ordered star product, and it is also a
deformation of the classical BRST operator $D$.  Then
the \emph{standard ordered BRST algebra}
$(\mathcal{A}^{(\bullet)}[[\lambda]], \star_\std,
\boldsymbol{D}_\std)$ becomes a differential $\mathbb{Z}$-graded
algebra with unit over $\mathbb{C}[[\lambda]]$.

The standard ordered quantum BRST operator splits into two
differentials:
\begin{itemize}
\item The \emph{quantized Chevalley-Eilenberg differential}
  $\boldsymbol{\delta}\colon \mathcal{A}^{\bullet,\bullet}[[\lambda]]
  \longrightarrow \mathcal{A}^{\bullet +1,\bullet}[[\lambda]]$,
  i.e. the Chevalley-Eilenberg differential on the quantum BRST
  complex induced by the quantum representation
  \begin{equation}
    \label{eq:gRepQuantumKoszul}
    \liealg{g} \ni \xi
    \longmapsto
    \boldsymbol{\rho}(\xi)
    =
    \ad(\xi)\otimes \id + \id \otimes \boldsymbol{\rho}_M(\xi),
  \end{equation}
  where $\boldsymbol{\rho}_M$ is the representation of $\liealg{g}$ on
  $\Cinfty(M)[[\lambda]]$ as in \eqref{eq:gRepQuantumAlgM}.
\item The \emph{quantized Koszul differential} $\boldsymbol{\del}
  \colon \mathcal{A}^{\bullet,\bullet}[[\lambda]] \longrightarrow
  \mathcal{A}^{\bullet ,\bullet-1}[[\lambda]]$ defined by
  \begin{equation}
    \label{eq:QuantizedKoszulDifferential}
    \boldsymbol{\del}(x \otimes f)
    =  \ins(\basis{e}^a) x \otimes f \star \boldsymbol{J}_a
    + \frac{\I \lambda}{2}
    \left(
    f^b_{ab} \ins(\basis{e}^a)
    + f^c_{ab} \basis{e}_c \wedge \ins(\basis{e}^a)\ins(\basis{e}^b)
    \right)
    \left(x\otimes f\right)
  \end{equation}
  for $x \in \Anti^\bullet (\liealg{g}_\mathbb{C}^*\oplus
  \liealg{g}_\mathbb{C})[[\lambda]]$ and $f \in
  \Cinfty(M)[[\lambda]]$. Note that the definition is independent of the basis.
\end{itemize}
By the equivariance of $\star$ we have $\boldsymbol{\rho} = \rho$ and
hence the equality $\boldsymbol{\delta}=\delta$. Moreover, as in the
classical case one has the splitting
\begin{equation}
  \label{eq:StandardBRSTsplitting}
  \boldsymbol{D}_\std
  =
  \boldsymbol{\delta} + 2 \boldsymbol{\del},
\end{equation}
see \cite[Thm.~17]{bordemann.herbig.waldmann:2000a} for further
details. The corresponding cohomology $\ker\boldsymbol{D}_\std /
\image \boldsymbol{D}_\std$ of the standard ordered quantum BRST
algebra is denoted by $\boldHBRST^{(\bullet)}(\mathcal{A}[[\lambda]])$
and called \emph{quantum BRST cohomology}.
Similarly to the classical setting, the quantum BRST cohomology is a
$\mathbb{Z}$-graded associative algebra with
$\mathbb{C}[[\lambda]]$-bilinear product $\star_\std$ induced by the
associative multiplication $\star_\std$ of the quantum BRST
algebra. One has $[a]\star_\std[b]=[a\star_\std b]$ for all
$[a],[b]\in \boldHBRST^{(\bullet)}(\mathcal{A}[[\lambda]])$ with
$\boldsymbol{D}_\std a = 0 = \boldsymbol{D}_\std b$ and $[1]\in
\boldHBRST^{(\bullet)}(\mathcal{A}[[\lambda]])$ is a unit with respect
to $\star_\std$.

Finally we recall that there exists a \emph{deformed restriction map}
\begin{equation}
  \boldsymbol{\iota^*}
  =
  \iota^* \circ S
  =
  \sum_{r=0} \lambda^r \boldsymbol{\iota^*}_r
  \colon
  \Anti^\bullet\liealg{g}^* \otimes \Cinfty(M)[[\lambda]]
  \longrightarrow
  \Anti^\bullet\liealg{g}^* \otimes \Cinfty(C)[[\lambda]],
\end{equation}
uniquely determined by the properties
\begin{align}
  \boldsymbol{\iota^*}_0
  =
  \iota^*,
  \quad
  \boldsymbol{\iota^*} \boldsymbol{\del}\at{\mathcal{A}^{\bullet,1}[[\lambda]]}
  =
  0
  \quad
  \text{and}
  \quad
  \boldsymbol{\iota^*} \prol
  =
  \id_{\Anti^\bullet\liealg{g}^* \otimes \Cinfty(C)[[\lambda]]}.
\end{align}
Here $S= \id_{\Cinfty(M)} + \sum_{r=1}^\infty \lambda^r S_r$ is a
formal series of differential linear operators of $\Cinfty(M)$ with
$S_r$ vanishing on constants. If the action of $\group{G}$ is in
addition proper on $M$ then $S$ can be chosen to be
$\group{G}$-equivariant.
Extending $\boldsymbol{\iota^*}$ by zero to the whole BRST algebra
$\mathcal{A}^{(\bullet)}[[\lambda]]$ one gets the following result, see
\cite[Prop.~26, Thm.~29, Thm.~32]{bordemann.herbig.waldmann:2000a}
for a proof and further details.

\begin{proposition}
  \label{prop:QuantumBRSTCohomology}
  Let $(M,\star,\group{G},\boldsymbol{J},C)$ be a Hamiltonian quantum
  $\group{G}$-space with regular constraint surface and proper action
  on $M$.
  \begin{enumerate2}
  \item There exists a $\group{G}$-equivariant chain homotopy
    $\widehat{\boldsymbol{h}}$ for the augmented standard ordered BRST
    operator
    \begin{equation}
      \label{eq:AugentedBRSTOperator}
      \widehat{\boldsymbol{D}}_\std = \boldsymbol{D}_\std +
      \boldsymbol{\delta}^c +2\boldsymbol{\iota^*} \in\End \left(
      \left(\Anti^\bullet\liealg{g}^*\otimes
      \Cinfty(C)[[\lambda]]\right) \oplus \mathcal{A}[[\lambda]]
      \right),
    \end{equation}
    with $\boldsymbol{\delta}^c$ being the Chevalley-Eilenberg
    differential on $\Anti^\bullet\liealg{g}^*\otimes
    \Cinfty(C)[[\lambda]]$, where all maps are defined to be zero on
    the domains on which they were previously not defined.  In
    particular, one has $\widehat{\boldsymbol{D}}_\std
    \widehat{\boldsymbol{h}}
    +\widehat{\boldsymbol{h}}\widehat{\boldsymbol{D}}_\std = 2\id$
    with $\widehat{\boldsymbol{h}} = \prol + \boldsymbol{h}$ and
    \begin{equation}
      \boldsymbol{h} \colon
      \Anti^\bullet \liealg{g}^* \otimes \Anti^\bullet  \liealg{g}
      \otimes \Cinfty(M)[[\lambda]]
      \longrightarrow
      \Anti^\bullet \liealg{g}^* \otimes\Anti^{\bullet+1}  \liealg{g}
      \otimes \Cinfty(M)[[\lambda]].
    \end{equation}
  \item The $\mathbb{C}[[\lambda]]$-linear map
    \begin{equation}
      \label{eq:QuantumIsoBRSTCE}
      \Psi \colon
      \boldHBRST^{(\bullet)}(\mathcal{A}[[\lambda]])
      \longrightarrow
      \boldHCE^\bullet(\liealg{g},\Cinfty(C)[[\lambda]])
      \cong
      \HCE^\bullet(\liealg{g},\Cinfty(C))[[\lambda]],
      \quad
	  [a]
	  \longmapsto
	      [\boldsymbol{\iota^*} a]
    \end{equation}
    is an isomorphism with inverse $
    \Psi^{-1}([c])=\left[\widehat{\boldsymbol{h}}c\right]$ for $[c]\in
    \boldHCE^\bullet(\liealg{g},\Cinfty(C)[[\lambda]])$.
  \item If the action is in addition free on $C$, then
    $\boldHBRST^{(0)}(\mathcal{A}[[\lambda]]) \cong \Cinfty(M_\red)[[\lambda]]$
    and this construction induces a star product $\star_\red$ on
    $M_\red$ via
    \begin{equation}
      \label{eq:ReducedStarProduct}
      \pi^*(u_1\star_\red u_2)
      =
      \boldsymbol{\iota^*} (\prol(\pi^*u_1)\star \prol(\pi^* u_2))
    \end{equation}
    for all $u_1,u_2\in\Cinfty(M_\red)[[\lambda]]$, the so-called
    \emph{reduced star product}.
	\end{enumerate2}
\end{proposition}
To shorten the notation we call $\mathcal{A}_\red =
\boldHBRST^{(0)}(\mathcal{A}[[\lambda]]) $ \emph{reduced quantum BRST
  algebra}.

%
%
\section{Abstract BRST Algebras and Different Types of Involutions}
\label{section:AbstractBRSTalgebras}

%
%
\subsection{Abstract BRST Algebras}
\label{subsection:AbstractBRSTalgebras}

Let $\ring{R}$ be an ordered ring with $\mathbb{Q}\subseteq \ring{R}$
and $\ring{C}=\ring{R}(\I)$ its complexification with $\I^2=-1$.
The main example is $\ring{R} = \mathbb{R}[[\lambda]]$, see
\cite{bordemann.waldmann:1998a,bursztyn.waldmann:2001a,bursztyn.waldmann:2005b}
for a detailed discussion on *-representations and the GNS construction in this
abstract setting.
In the following, $\mathcal{A}$ denotes a $\mathbb{Z}_2$-graded
associative algebra over $\ring{C}$ and $\ad(a)=[a,\argument]$ the super
commutator with respect to the $\mathbb{Z}_2$-grading.

\begin{definition}[BRST algebra]
  \label{defi:AbstractBRSTalgebra}
  Let $\mathcal{A}=\mathcal{A}_0 \oplus \mathcal{A}_1$ be a
  $\mathbb{Z}_2$-graded associative algebra over
  $\ring{C}=\ring{R}(\I)$.
  \begin{enumerate2}
   \item An even element $\gamma \in \mathcal{A}_0$ such that the
     inner derivation $\Gh = \ad(\gamma)=[\gamma,\argument]$ induces a
     $\mathbb{Z}$-grading on $\mathcal{A}$ by
         \begin{equation}
           \mathcal{A}^{(\bullet)}
	   =
	   \bigoplus_{k\in \mathbb{Z}} \mathcal{A}^{(k)}
           \quad
	   \text{with}
	   \quad
           \mathcal{A}^{(k)}
	   =
	   \{a\in\mathcal{A}\mid \Gh a = ka \}
         \end{equation}
         is called \emph{ghost charge}. The operator $\Gh$ is called
         \emph{ghost number operator} and the induced grading is
         called \emph{ghost number grading}.

       \item An odd element $\Theta$ with ghost number $+1$ and square
         zero, i.e.
	 \begin{equation}
	   \Theta \in \mathcal{A}^{(1)}_1
	   \quad
	   \text{and}
	   \quad
	   \Theta^2
	   =
	   0,
	 \end{equation}
         is called \emph{BRST charge}. The corresponding inner
         derivation $D =\ad(\Theta)$ is called \emph{BRST operator}.
  \end{enumerate2}
  The triple $(\mathcal{A},\gamma,\Theta)$ is then called \emph{BRST
    algebra} over $\ring{C}$. A \emph{morphism} $\Phi\colon
  (\mathcal{A},\gamma,\Theta) \longrightarrow
  (\mathcal{A}',\gamma',\Theta')$ of BRST algebras is an even morphism
  of $\mathbb{Z}_2$-graded associative algebras
  $\Phi\colon\mathcal{A}\longrightarrow \mathcal{A}'$ with
  \begin{equation}
    \label{eq:MorphismBRSTAlgebras}
    \Phi(\gamma)
	  =
		\gamma'
    \quad
	  \text{and}
	  \quad
    \Phi(\Theta)
	  =
	  \Theta',
  \end{equation}
  and the category of BRST algebras is denoted by $\BRSTalgebras$.
\end{definition}

Note that the properties imply that $\Phi$ preserves the
$\mathbb{Z}$-grading as well. We often
encounter the setting that the $\mathbb{Z}_2$-grading is induced by
the $\mathbb{Z}$-grading. In addition,
$D = \ad(\Theta) \colon \mathcal{A}^{(\bullet)}
\longrightarrow \mathcal{A}^{(\bullet +1)}$ and $D^2=0$ imply that
the BRST operator is a coboundary operator, thus it defines a cohomology:

\begin{definition}[BRST cohomology]
  \label{definition:AbstractBRSTCohomology}
  Let $(\mathcal{A},\gamma,\Theta)$ be a BRST algebra. Then
  \begin{equation}
    \HBRST^{(\bullet)} (\mathcal{A})
	  =
	  \bigoplus_{k\in\mathbb{Z}} \HBRST^{(k)} (\mathcal{A})
    \quad
	  \text{with}
	  \quad
    \HBRST^{(k)} (\mathcal{A})
    =
	  \frac{\ker D\at{\mathcal{A}^{(k)}}}
	     {\image D\at{\mathcal{A}^{(k-1)}}}
  \end{equation}
  is called \emph{BRST cohomology}  of $\mathcal{A}$.
  The \emph{reduced BRST algebra} is defined by
  \begin{equation}
    \label{eq:AbstractReducedBRSTAlgebra}
    \mathcal{A}_\red
	  =
	  \mathrm{H}^{(0)}_{\BRST,0}(\mathcal{A}).
  \end{equation}
\end{definition}
Since $D$ is an odd inner derivation, the cohomology is again a
$\mathbb{Z}\times\mathbb{Z}_2$-graded associative algebra and
$\mathcal{A}_\red$ is a well-defined associative subalgebra.  The
ghost number operator acts on $\HBRST^{(\bullet)} (\mathcal{A})$ via
\begin{equation}
  \Gh_\BRST [a]
  =
  [\Gh a],
\end{equation}
i.e. $\HBRST^{(k)} (\mathcal{A})= \{[a]\in \HBRST^{(\bullet)} (\mathcal{A}) \mid
\Gh_\BRST[a] = k[a]\} $. However, it is no longer an
inner derivation as $\gamma$ is no cocycle
\begin{equation}
  D \gamma
  =
  [\Theta,\gamma]
  =
  - [\gamma,\Theta]
  =
  - \Theta.
\end{equation}

If the $\mathbb{Z}_2$-grading is induced by the $\mathbb{Z}$-grading,
then we have $\mathcal{A}_\red = \HBRST^{(0)}(\mathcal{A})$ as in the
case of the quantum BRST cohomology. A straightforward computation
shows that the assignment of a BRST algebra
$(\mathcal{A},\gamma,\Theta)$ to its BRST cohomology
$\HBRST^{(\bullet)}(\mathcal{A})$ and reduced BRST algebra
$\mathcal{A}_\red$ is a functor from $\BRSTalgebras$ into the category
of $\mathbb{Z}\times \mathbb{Z}_2$-graded algebras resp. algebras.

Let us consider a *-involution for $\mathcal{A}$.
Since we aim to get an induced involution on $\mathcal{A}_\red=
\mathrm{H}^{(0)}_{\BRST,0}(\mathcal{A})$, the involution on the whole
of $\mathcal{A}$ should respect the $\mathbb{Z}_2$-grading. We have
two main possibilities for involutions on a $\mathbb{Z}_2$-graded
algebra $\mathcal{A}$:
\begin{itemize}
  \item \emph{Graded *-involutions} $I\colon \mathcal{A}
	      \longrightarrow \mathcal{A}$, i.e. $\ring{C}$-antilinear
		    involutive even maps with
        \begin{equation}
		      \label{eq:DefiGradedInvolution}
		      I(ab)
		      =
		      I(b)I(a)
        \end{equation}
        for all $a,b\in \mathcal{A}$. The pair $(\mathcal{A},I)$ is called
		    \emph{graded *-algebra}.

  \item \emph{Super *-involutions} $S\colon \mathcal{A}
	      \longrightarrow \mathcal{A}$, i.e. $\ring{C}$-antilinear
		    involutive even maps with
        \begin{equation}
		      \label{eq:DefiSuperInvolution}
		      S(ab)
		      =
		      (-1)^{\abs{a}\abs{b}}  S(b) S(a)
        \end{equation}
        for all homogeneous elements $a,b\in \mathcal{A}$ with
		    degrees $\abs{a},\abs{b}$. The pair $(\mathcal{A},S)$
		    is called \emph{super *-algebra}.
\end{itemize}
A short computation shows that the graded resp. super *-involutions of
the adjoint representations give a minus sign.  This motivates the
following rescaling: From now on $\gamma\in\mathcal{A}^{(0)}_0$ and
$\Theta\in\mathcal{A}^{(1)}_1$ are the elements such that
\begin{equation}
  \label{eq:RescaledGhostBRST}
  \Gh
  =
  \I\ad(\gamma)
  \quad
  \text{and}
  \quad
  D
  =
  \I \ad(\Theta).
\end{equation}
Note that the normalization does not change the cohomology of $D$ as
well as the grading induced by $\Gh$ and that in the case of the
quantum BRST algebra we already have a corresponding factor
$\frac{1}{\I }$ in front of the super commutator.

One can show that the notion of super and graded *-involutions can be
mutually exchanged by rescaling the odd component of the involution
by $\pm \I$. Thus it only remains to investigate possible
compatibilities of involutions with the ghost number grading. As we
ultimately want an induced *-involution on the even ghost number zero
part of the BRST cohomology, the ghost number zero part should be
invariant under the involution, too. There are again \emph{two main
possibilities}:
An involution that leaves the ghost number grading invariant, or
an involution that inverts the ghost number grading.

\begin{remark}[Involution leaving ghost number invariant]
  A super *-involution that leaves the ghost number degree invariant and
  with Hermitian BRST charge $\Theta^* = \Theta$ induces a super
	*-involution on the cohomology and a *-involution on $\mathcal{A}_\red$
	in a functorial way. However, this kind of involution has a big
	disadvantage in connection with *-representations $\pi$ on pre-Hilbert
	spaces over $\ring{C}$, see \cite{bursztyn.waldmann:2001a}:
	In this case $\pi(\Theta)=0$ would vanish.
	The induced inner product on the physical space is in general still
	not positive definite, which leads to so-called no ghost theorems,
	compare e.g. \cite[Sect.~14.2]{henneaux.teitelboim:1992a}.
\end{remark}

 The above remark is a consequence of a more general problem:

\begin{remark}
  The theory of homological algebra is not compatible with star involutions resp.
  with the positive definiteness of inner products. In the case of three
	or more dimensions there are no canonically induced inner products on
	the cohomology as the following simple example shows:
	Consider $\mathbb{R}^3$ with Euclidean scalar product and
	differential
	\begin{equation*}
	  \D
		=
		\left(
		\begin{smallmatrix}
		  0 & 0 & 1 \\
			0 & 0 & 0 \\
			0 & 0 & 0
		\end{smallmatrix}
		\right).
	\end{equation*}
	Then $\ker \D / \image \D \cong \Span\{ ( 0 \; 1 \; 0)^\top\}$ and
	the quotient map is not compatible with the inner product.
\end{remark}

We are mainly interested in non-trivial *-representations of the
reduced quantum BRST algebra with involution, where one of the motivations
consists in implementing the superposition principle.
Therefore, the above lack of positivity leads us to the study of
other possibilities for involutions $^*$ on the BRST algebra
$\mathcal{A}$ such that $D$ and $\Theta$ are not Hermitian,
i.e.  $D \neq D^*$ and $\Theta \neq \Theta^*$.

%
%
\subsection{Graded *-Involution with Imaginary Ghost Operator}
\label{subsection:GradedInvolutionwithImaginaryGhost}

Consider a BRST algebra $(\mathcal{A},\gamma,\Theta)$ which has an
additional graded *-involution $a\mapsto a^*$. Since super and graded
*-involutions can be mutually exchanged, this is only a matter of
convenience and no relevant choice.

\begin{definition}[BRST *-algebra with imaginary ghost operator]
	\label{defi:BRSTstarAlgebrawithImaginaryGhost}
	Let $\mathcal{A}$ be a BRST algebra together with a graded
	*-involution $^*$ satisfying $\Gh^* =\I \ad(\gamma^*)= - \Gh$. Then one calls
	$(\mathcal{A},\gamma,\Theta,^*)$ \emph{BRST *-algebra with
	imaginary ghost operator}. A \emph{morphism}
	$\Phi\colon (\mathcal{A},^*)\longrightarrow (\mathcal{B},^*)$
	of BRST *-algebras with imaginary ghost 	operators is a
	morphism $\Phi\colon\mathcal{A}\longrightarrow\mathcal{B}$ of
	BRST algebras that fulfils
	\begin{equation}
	  \Phi(a^*)
	  =
	  \Phi(a)^*
	\end{equation}
	for all $a\in \mathcal{A}$. The corresponding category of BRST
        *-algebras with imaginary ghost operators is denoted by
        $\iBRSTalg$.
\end{definition}

Since the graded *-involution is compatible with the
$\mathbb{Z}_2$-grading and since it inverts the ghost number grading,
we directly see that $\mathcal{A}^{(0)}$ becomes a
$\mathbb{Z}_2$-graded *-subalgebra of $\mathcal{A}$.  Similarly,
$\mathcal{A}^{(0)}_0$ becomes a *-subalgebra of $\mathcal{A}$.
Moreover, we obtain the following behaviour of the ghost charge and
the BRST operator under the graded *-involution.

\begin{lemma}
  Let $(\mathcal{A},^*)$ be a BRST algebra with imaginary ghost
  operator.
	\begin{enumerate2}
	\item There exists a unique central Hermitian
        element $c\in \mathcal{A}^{(0)}_0$ such that one has
        \begin{equation}
          \gamma^*
	        =
	        -\gamma +c.
        \end{equation}

	\item Define the \emph{adjoint BRST operator} $D^* \colon
	      \mathcal{A}^{(\bullet)}\longrightarrow \mathcal{A}^{(\bullet-1)}$
        by $D^*=\I\ad(\Theta^*)$. Then one has $(D^*)^2=0$ and
        \begin{equation}
          \label{eq:Dstara}
          D^* a
	        =
	        (-1)^{\abs{a}}(Da^*)^*
        \end{equation}
        for homogeneous $a\in \mathcal{A}$ with degree $\abs{a}$.
	\end{enumerate2}
\end{lemma}
\begin{proof}
  Concerning the first point we have $ -\I\ad(\gamma) = \I
  \ad(\gamma^*)$ and thus $\gamma^* + \gamma$ is in the center of
  $\mathcal{A}$ as well as $\gamma,\gamma^* \in \mathcal{A}^{(0)}_0$,
  hence the statement is shown.  For the second point note that
  $\Theta\in\mathcal{A}^{(1)}_1$, so $\Theta^* \in
  \mathcal{A}^{(-1)}_1$, and \eqref{eq:Dstara} follows from a short
  computation.
\end{proof}

The element $\Delta = \Theta\Theta^* + \Theta^*\Theta = \Delta^*
 \in \mathcal{A}^{(0)}_0$ is called \emph{Laplacian} and will play
an important role in the representation theory.
It follows that $\Theta$ and $\Theta^*$ are either linearly
independent or both equal to zero as $\mathcal{A}^{(1)} \cap
\mathcal{A}^{(-1)} = \{0\}$.  Thus the kernel of $D$ is no
*-subalgebra of $\mathcal{A}$ and there is no obvious way to obtain a
*-structure on the BRST cohomology $\HBRST^{(\bullet)}(\mathcal{A}) =
\ker D/\image D$ of $\mathcal{A}$ or at least on the reduced
algebra. Therefore, the idea is to define a new quotient $(\ker D\cap
\ker D^*) / (\image D\cap \image D^*)$ and to show that this
construction yields a well-defined
$\mathbb{Z}\times\mathbb{Z}_2$-graded algebra with graded
*-involution.  To this end we investigate the relation between the
*-involution and the elements in $\ker D\cap \ker D^*$ and $\image D
\cap \image D^*$.

\begin{lemma}
  \label{lemma:SectionImagesStarSubideal}
  Let $(\mathcal{A},^*)$ be a BRST *-algebra with imaginary ghost
  operator. Then one has for all $a\in\mathcal{A}$:
	\begin{enumerate2}
		\item $ a\in \ker D\cap \ker D^*
		      \quad
				  \Longleftrightarrow
				  \quad
	        a,a^* \in \ker D
				  \quad
				  \Longleftrightarrow
				  \quad
		      a,a^* \in \ker D^*. $

		\item $	a\in \image  D\cap \image D^*
		      \quad \; \,
				  \Longleftrightarrow
				  \quad
		      a,a^* \in \image D
				  \quad \;
				  \Longleftrightarrow
				  \quad
		      a,a^* \in \image D^*$.
	\end{enumerate2}
  Consequently, the intersection $\ker D\cap \ker D^*$ is a
  $\mathbb{Z}\times\mathbb{Z}_2$-graded *-subalgebra of $\mathcal{A}$
  and the set $\image D\cap \image D^* \subseteq \ker D\cap \ker D^*$
  is a $\mathbb{Z}\times\mathbb{Z}_2$-graded *-ideal therein.
\end{lemma}
\begin{proof}
  The first two parts follow directly with \eqref{eq:Dstara}. In
  addition, we have for all homogeneous elements $a\in \ker D\cap \ker
  D^*$ and $De = D^* f \in \image D\cap \image D^*$
  \begin{align*}
    a De
    =
    (-1)^{\abs{a}} D(ae)
    =
    (-1)^{\abs{a}} D^*(af),
  \end{align*}
  thus $\image D\cap \image D^*$ is a *-ideal in $\ker D\cap \ker
  D^*$.
\end{proof}

Hence we know that $(\ker D\cap \ker D^*) / (\image D\cap \image D^*)$
becomes a $\mathbb{Z}\times \mathbb{Z}_2$-graded algebra as well.

\begin{definition}[Reduced BRST *-algebra]
  Let $(\mathcal{A},^*)$ be a BRST *-algebra with imaginary ghost operator.
  The \emph{BRST quotient} is defined by
  \begin{equation}
    \label{eq:BRSTQuotient}
    \HBRSTtilde^{(\bullet)}(\mathcal{A})
    =
    \frac{\ker D\cap \ker D^*}{\image D\cap \image D^*},
  \end{equation}
  and by
  \begin{equation}
    \label{eq:ReducedBRSTstarAlgebra}
    \widetilde{\mathcal{A}}_\red
    =
    \widetilde{\mathrm{H}}^{(0)}_{\BRST,0}(\mathcal{A})
  \end{equation}
  one denotes the corresponding \emph{reduced BRST *-algebra}.
\end{definition}
Note that $\HBRSTtilde^{(\bullet)}(\mathcal{A})$ can in general
\emph{not} be expressed as cohomology of some cohomological chain
complex since it is only a quotient of an algebra with an ideal.
Nonetheless, we sometimes call it cohomology in analogy to
$\HBRST^{(\bullet)}(\mathcal{A})$ and to simplify the notation.
We have the following result.

\begin{lemma}
  \label{lemma:ReducedBRSTStarAlgebra}
  The BRST quotient $\HBRSTtilde^{(\bullet)}(\mathcal{A})$ is a
  $\mathbb{Z}\times \mathbb{Z}_2$-graded algebra with
  graded *-involution $^*$ defined by
  \begin{equation}
    [a]^*
    =
    [a^*],
  \end{equation}
  exchanging $\HBRSTtilde^{(k)}(\mathcal{A})$ with
  $\HBRSTtilde^{(-k)}(\mathcal{A})$ for all $k\in\mathbb{Z}$.
  The reduced BRST *-algebra $\widetilde{\mathcal{A}}_\red $
  is a *-algebra.
\end{lemma}
\begin{proof}
  The properties follow directly by the above results and the
  compatibility of the *-involution with the grading.
\end{proof}

Just as for BRST algebras one shows that the passages from
a BRST *-algebra with imaginary ghost operator to its BRST quotient
and reduced BRST *-algebra are functorial:

\begin{proposition}
  The assignment of a BRST *-algebra with imaginary ghost operator
  $(\mathcal{A},\gamma,\Theta,^*)$ to the BRST quotient
  $\HBRSTtilde^{(\bullet)}(\mathcal{A})$ is a functor into the category
  of $\mathbb{Z}\times\mathbb{Z}_2$-graded algebras with graded
  *-involution. Similarly, the assignment to the reduced BRST
  *-algebra $\widetilde{\mathcal{A}}_\red$ is a functor into the
  category of *-algebras.
\end{proposition}

Finally, we can prove that there is the following crucial relation between
$\HBRSTtilde^{(\bullet)}(\mathcal{A})$ and
$\HBRST^{(\bullet)}(\mathcal{A})$.

\begin{proposition}
  \label{prop:InclBRSTQuotienttoCohomology}
  The map
  \begin{equation}
    I_\mathcal{A} \colon
    \HBRSTtilde^{(\bullet)}(\mathcal{A})
    \longrightarrow
    \HBRST^{(\bullet)}(\mathcal{A}),
    \quad
    [a]
    \longmapsto
    I_\mathcal{A}([a]) = [a]
  \end{equation}
  is a well-defined  morphism of
  $\mathbb{Z}\times\mathbb{Z}_2$-graded algebras.
\end{proposition}
\begin{proof}
The well-definedness follows directly with the definitions
of the quotients and the compatibility with the grading is
clear as both $\HBRSTtilde^{(\bullet)}(\mathcal{A})$ and
$\HBRST^{(\bullet)}(\mathcal{A})$ inherit the
$\mathbb{Z}\times\mathbb{Z}_2$-grading of $\mathcal{A}$.
\end{proof}

\begin{remark}
  The important question is if this canonical morphism $I_\mathcal{A}$
	is an isomorphism, which would justify our construction and yield a
	canonical involution on the BRST cohomology. In general,
	there seems to be no possibility to decide whether $I_\mathcal{A}$ is
	injective or surjective and one has to argue which reduction scheme
	fits better to the respective application. In
	Section~\ref{subsection:ComparisonofReducedQuantumBRSTAlgebras} we show
	that in our example of the quantum BRST algebra $I_\mathcal{A}$ is an
	isomorphism if restricted to the physically most relevant zero-th degree.
\end{remark}

\begin{remark}
The above considerations show that in some cases it might be useful
to consider instead of the usual cohomology the BRST quotient from
\eqref{eq:BRSTQuotient}. To further justify this proposal one has to transfer
the concepts of quasi-isomorphisms and chain homotopies from homological algebra
to our setting. For the notion of quasi-isomorphisms there is an obvious choice:
We call a morphism $\Phi\colon\mathcal{A}\longrightarrow\mathcal{B}$ of BRST
*-algebras \emph{quasi-isomorphism} if it induces an isomorphism
$\Phi \colon \HBRSTtilde^{(\bullet)}(\mathcal{A}) \rightarrow
\HBRSTtilde^{(\bullet)}(\mathcal{B})$ on the BRST quotients. The case of
chain homotopies is more subtle: One choice would be to consider a homotopy
$h \colon \mathcal{A}^{(\bullet)} \rightarrow \mathcal{B}^{(\bullet -1)}$ between
two morphisms $\Phi,\Psi \colon \mathcal{A} \rightarrow \mathcal{B}$ of
BRST *-algebras with respect to the BRST operator $D$, i.e. a $\ring{C}$-linear
map $h$ such that
\begin{equation*}
  h D_\mathcal{A} + D_\mathcal{B} h
	=
	\Phi - \Psi.
\end{equation*}
Then the map $h^* \colon \mathcal{A}^{(\bullet)} \rightarrow \mathcal{B}^{(\bullet +1)}$,
given on homogeneous elements $a\in \mathcal{A}$ by
$h^*(a) = - (-1)^{\abs{a}} (h(a^*))^*$
turns out to be a chain homotopy with respect to $D^*$ between $\Phi$ and $\Psi$.
In particular, in this case $\Phi$ and $\Psi$ induce the same maps on the
BRST quotients. However, it is not yet clear to us if this is the compatibility we
want to have and we plan to investigate it in a forthcoming paper.
\end{remark}

In the remaining part of this section we want to show that
*-involutions with imaginary ghost operators lead indeed to a non-trivial
*-representation theory on pre-Hilbert spaces, in contrast to
the involutions with Hermitian BRST charges and Hermitian ghost operators.

%
%
\subsection{BRST *-Representations and GNS Construction}
\label{subsection:BRSTStarRep}

We introduce a *-representation theory of BRST *-algebras
with imaginary ghost operator and show that the representations can be
reduced to *-representations of the reduced BRST *-algebras. In addition,
we sketch an adapted GNS construction. The notions are based on the
theory of pre-Hilbert spaces $\prehilb{H}$ as in
\cite[Chapter~7]{waldmann:2007a}. Recall that a pre-Hilbert space over
$\ring{C}$ is a $\ring{C}$-module $\prehilb{H}$ with positive definite
inner product $\SP{\argument,\argument}$. Note that positivity, i.e.
$\SP{\phi,\phi} > 0$ for all $\phi \in \prehilb{H}\setminus \{0\}$, makes sense
in our setting as $\ring{R} \subset \ring{C} = \ring{R}(\I)$ is ordered. A map
$A\colon \prehilb{H} \rightarrow \prehilb{H}$ is called adjointable if
there exists a map $A^* \colon \prehilb{H} \rightarrow \prehilb{H}$ such
that $\SP{A\phi, \psi}	=	\SP{\phi,A^* \psi}$
for all $\phi,\psi \in \prehilb{H}$. The set of adjointable maps
is denoted by $\Bounded(\prehilb{H})$.
These spaces can be adapted to our setting:

\begin{definition}[BRST pre-Hilbert space]
  A \emph{BRST pre-Hilbert space} is a $\mathbb{Z}\times \mathbb{Z}_2$-graded
	$\ring{C}$-module $\prehilb{H}$ together with
	\begin{enumerate2}
	  \item an odd endomorphism $\Theta_\prehilb{H} \in \End_1^{(1)}(\prehilb{H})$
		      of ghost number degree $+1$ with $\Theta_\prehilb{H}^2 = 0$, called
					\emph{BRST operator},
		\item a \emph{ghost number operator} $\gamma_\prehilb{H}$ defined by
		      \begin{equation}
            \I \gamma_\prehilb{H}\at{\prehilb{H}^{(k)}}
	          =
	          k \cdot\id\at{\prehilb{H}^{(k)}}
          \end{equation}
          and extended $\ring{C}$-linearly to all of $\prehilb{H}$, and
		\item an even graded positive definite inner product
		      $\SP{\argument,\argument}$,
	\end{enumerate2}
  such that one has the compatibilities
    \begin{equation}
        \label{eq:CompatibilityGhosts}
	  \gamma^*_\prehilb{H}
		=
		-\gamma_\prehilb{H}
		\quad \quad \text{ and } \quad \quad
		\Theta_\prehilb{H}
		\in \Bounded(\prehilb{H}).
	\end{equation}
	A morphism $T\colon \prehilb{H}\longrightarrow \prehilb{H}'$
	of BRST pre-Hilbert spaces is an adjointable $\ring{C}$-linear even
	map intertwining the BRST and ghost operators.
\end{definition}

\begin{remark}
  Alternatively, one could also consider isometric maps as morphisms of
	BRST pre-Hilbert spaces, instead of adjointable ones. The isomorphisms in
	this category are then unitary intertwiners, not adjointable bijective
	intertwiners. In our general setting these notions lead indeed to different
	notions of equivalent representations, and we favour adjointable ones because
	there might exist isometric maps not allowing for an adjoint.
\end{remark}

Note that the definition directly implies that $\Bounded^{(\bullet)}(\prehilb{H})
= \Bounded(\prehilb{H}) \cap \End^{(\bullet)}(\prehilb{H})$ is a well-defined
BRST *-algebra with imaginary ghost operator.
As in the case of BRST *-algebras one can construct the usual BRST cohomology
$\HBRST^{(\bullet)}(\prehilb{H}) 	= \ker \Theta_\prehilb{H} /
\image\Theta_\prehilb{H}$, but there exists no canonical inner product on this
quotient. Therefore, we define the BRST quotient
\begin{equation}
  \HBRSTtilde^{(\bullet)}(\prehilb{H})
	=
	\frac{\ker \Theta_\prehilb{H} \cap \ker \Theta^*_\prehilb{H}}
  {\image\Theta_\prehilb{H}\cap \ker \Theta^*_\prehilb{H}}.
\end{equation}
One can directly check that this quotient is again $\mathbb{Z} \times
\mathbb{Z}_2$-graded and noting $\ker \Theta_\prehilb{H}^* =
(\image \Theta_\prehilb{H})^\bot$
we can even show more:

\begin{proposition}
  For a BRST pre-Hilbert space $\prehilb{H}$ one has $\image \Theta_\prehilb{H}
  \cap (\image\Theta_\prehilb{H})^\bot = \{0\}$ and thus
  \begin{equation}
    \HBRSTtilde^{(\bullet)}(\prehilb{H})
    =
		\ker \Theta_\prehilb{H} \cap \ker \Theta^*_\prehilb{H}
    =
		\ker\Delta_\prehilb{H},
  \end{equation}
	where $\Delta_\prehilb{H} = \left[\Theta_\prehilb{H},
	\Theta^*_\prehilb{H}\right]$ denotes the Laplacian of
	$\Bounded^{(\bullet)}(\prehilb{H})$. In particular, the inner product on
	$\prehilb{H}$ restricts to a positive definite and non-degenerate inner
	product on $\HBRSTtilde^{(\bullet)}(\prehilb{H})$.
\end{proposition}
\begin{proof}
One has $\image \Theta_\prehilb{H} \cap \ker \Theta^*_\prehilb{H} =
\image \Theta_\prehilb{H} \cap (\image\Theta_\prehilb{H})^\bot = \{0\}$
and thus $\HBRSTtilde^{(\bullet)}(\prehilb{H})  =
\ker \Theta_\prehilb{H} \cap \ker \Theta^*_\prehilb{H}$.
The positive definiteness of the inner product on $\prehilb{H}$ implies
for all $\phi\in \prehilb{H}$
\begin{align*}
  \SP{\phi, \Delta_\prehilb{H} \phi }
  =
	\SP{\Theta_\prehilb{H}^* \phi,\Theta_\prehilb{H}^*\phi }
  + \SP{\Theta_\prehilb{H} \phi,\Theta_\prehilb{H}\phi} \geq 0,
\end{align*}
which entails $\ker \Theta_\prehilb{H} \cap \ker \Theta^*_\prehilb{H}
=\ker\Delta_\prehilb{H}$. Hence the inner product on $\prehilb{H}$ restricts
to $\HBRSTtilde^{(\bullet)}(\prehilb{H})$. It is positive definite and in particular non-degenerate.
\end{proof}

Elements in the kernel of $\Delta_\prehilb{H}$ are called \emph{harmonic}.
The above proposition shows that in the case of a BRST pre-Hilbert space
$\prehilb{H}$ the BRST quotient $\HBRSTtilde^{(\bullet)}(\prehilb{H})$
and the \emph{reduced BRST pre-Hilbert space}
\begin{equation}
  \widetilde{\prehilb{H}}_\red
	=
	\widetilde{\mathrm{H}}^{(0)}_{\BRST,0}(\prehilb{H})
\end{equation}
inherit both positive definite inner products and one easily sees that all the passages
are functorial.  We have again a canonical
inclusion of the BRST quotient into the BRST cohomology, compare
Proposition~\ref{prop:InclBRSTQuotienttoCohomology}. Now it is even injective
by the positive definiteness:

\begin{proposition}
	Let $\prehilb{H}$ be a BRST pre-Hilbert space.
  Then the canonical map
	$I_\prehilb{H}\colon \HBRSTtilde^{(\bullet)}(\prehilb{H})
	\longrightarrow  \HBRST^{(\bullet)}(\prehilb{H})$ is injective.
\end{proposition}
\begin{proof}
Suppose $[\phi]= I_\prehilb{H}([\phi])= I_\prehilb{H}([\psi])=[\psi]$ for
$[\phi],[\psi]\in \HBRSTtilde^{(\bullet)}(\prehilb{H})$. This implies
$\phi-\psi = \Theta_\prehilb{H}\chi \in \image \Theta_\prehilb{H}
\cap  \ker \Theta_\prehilb{H} \cap \ker \Theta^*_\prehilb{H}$ and we can
compute
\begin{align*}
  \SP{\Theta_\prehilb{H}\chi,\Theta_\prehilb{H}\chi}
	=
	\SP{\chi, \Theta^*_\prehilb{H}\Theta_\prehilb{H}\chi}
	=
	0.
\end{align*}
But by the positive definiteness of the inner product this implies
$0 = \Theta_\prehilb{H}\chi = \phi -\psi$.
\end{proof}

\begin{remark}
  The above result is a more general version of the easy part of the
	Hodge theorem: the injectivity of the inclusion of the harmonic
	differential forms into the de Rham cohomology of a Riemannian manifold,
	see e.g. \cite[Lemma~4.15]{morita:2001a}. The difficult part is to show
	the surjectivity that does not always hold in our general situation.
\end{remark}

Now we can define BRST *-representations of BRST *-algebras with
imaginary ghost operators and their reduction:

\begin{definition}[BRST *-representation]
	\label{defi:brststarrep}
  Let $(\mathcal{A},\gamma,\Theta,^*)$ be a BRST *-algebra with imaginary
	ghost operator. A \emph{BRST *-representation} of $\mathcal{A}$ on a BRST
	pre-Hilbert space $\prehilb{H}$ is a morphism
  \begin{equation}
    \rho \colon
		\mathcal{A}
		\longrightarrow
		\Bounded^{(\bullet)}(\prehilb{H})
  \end{equation}
  of BRST *-algebras with imaginary ghost operator.
	An \emph{intertwiner} $T$ between two such BRST *-representations
	$(\prehilb{H},\rho)$ and $(\prehilb{H}',\rho')$ of $\mathcal{A}$ is a
	morphism $T \colon \prehilb{H} \longrightarrow \prehilb{H}'$
	of BRST pre-Hilbert spaces that satisfies in addition
  \begin{equation}
    T \circ \rho(a)
		=
		\rho'(a) \circ T
  \end{equation}
  for all $a \in \mathcal{A}$.
\end{definition}
Since all the passages from $\mathcal{A}$ to
$\HBRSTtilde^{(\bullet)}(\mathcal{A})$ and $\widetilde{\mathcal{A}}_\red$
as well as from $\prehilb{H}$ to $\HBRSTtilde^{(\bullet)}(\prehilb{H})$ and
$\widetilde{\prehilb{H}}_\red$ are functorial, we obtain the following
behaviour of BRST *-representations under the BRST reduction:

\begin{proposition}
	\label{lemma:ReductionofBRSTStarRep}
  Consider a BRST *-algebra $(\mathcal{A},\gamma,\Theta,^*)$ with imaginary
	ghost operator and a BRST *-representation $\rho$ of $\mathcal{A}$
	on a BRST pre-Hilbert space $\prehilb{H}$. Then
  \begin{equation}
    \widetilde{\rho}_\BRST \colon
		\HBRSTtilde^{(\bullet)}(\mathcal{A})
    \longrightarrow
    \Bounded^{(\bullet)}\left(\HBRSTtilde^{(\bullet)}(\prehilb{H})\right),
		\quad \quad
		\widetilde{\rho}_\BRST([a])[\phi]
		=
		[\rho(a)\phi]
  \end{equation}
  yields a *-representation of $\HBRSTtilde^{(\bullet)}(\mathcal{A})$ on
  $\HBRSTtilde^{(\bullet)}(\prehilb{H})$ which is compatible with all degrees.
	Moreover, the restriction
  \begin{equation}
    \widetilde{\rho}_\red
		=
		\left(\widetilde{\rho}_\BRST\at{\widetilde{\mathrm{H}}^{(0)}_{\BRST,0}
		(\mathcal{A})}\right)\at{\widetilde{\mathrm{H}}^{(0)}_{\BRST,0}(\prehilb{H})}
  \end{equation}
  yields a *-representation of
  $\widetilde{\mathcal{A}}_\red$ on $\widetilde{\prehilb{H}}_\red$.
	All the assignments are functorial.
\end{proposition}

Finally, we apply the general formalism of the GNS construction to the case of a BRST *-algebra with imaginary ghost operator, i.e. we construct
BRST *-representations out of suitable linear functionals.
We recall at first the usual GNS construction from
\cite[Section~7.2.2]{waldmann:2007a}:

\begin{remark}[GNS representation]
	\label{remark:gnsrepresentation}
	Let $\mathcal{A}$ be a *-algebra over $\ring{C}$ and $\omega\colon
	\mathcal{A}\longrightarrow \ring{C}$ a positive linear functional,
	i.e. $\omega(a^*a)\geq 0$ for all $a\in \mathcal{A}$. Then one has
	$\omega(b^*a) = \cc{\omega(a^* b)}$ for all $a,b\in\mathcal{A}$ as
	well as the Cauchy-Schwarz inequality. The subset
	\begin{equation*}
	  \mathcal{I}_\omega
		=
		\{ a\in\mathcal{A}\mid \omega(a^*a)=0\}
		=
		\{ a\in\mathcal{A} \mid \omega(b^*a)=0 \;\;\forall \; b\in \mathcal{A}\}
	  =
		\{ a\in\mathcal{A}\mid \omega(a^*b)=0 \;\; \forall \; b\in \mathcal{A}\}
	\end{equation*}
	is a left ideal in $\mathcal{A}$, the so-called \emph{Gel'fand ideal}.
	The quotient $\prehilb{H}_\omega =\mathcal{A}/\mathcal{I}_\omega$ becomes a
	left $\mathcal{A}$-module in the canonical way by setting
	$\pi_\omega(a)\psi_b=\psi_{ab}$ for $a,b\in\mathcal{A}$,
	where $\psi_b\in\prehilb{H}_\omega$ denotes the equivalence class of $b$.
	One has a positive definite inner product
	$\SP{\psi_a,\psi_b}_\omega=\omega(a^*b)$ on
	$\prehilb{H}_\omega$ and $\pi_\omega$ turns out to be a *-representation of
	$\mathcal{A}$, the so-called \emph{GNS representation} with respect to $\omega$.
\end{remark}

The representations of BRST *-algebras $\mathcal{A}$ should be
compatible with the $\mathbb{Z}\times\mathbb{Z}_2$-grading, whence we
have to consider $\mathbb{Z}\times\mathbb{Z}_2$-homogeneous positive
linear functionals $\omega \colon \mathcal{A} \rightarrow \ring{C}$, i.e.
such positive linear functionals that vanish on all degrees except
$\mathcal{A}^{(0)}_0$. In this case one easily sees that
$\mathcal{I}_\omega$ and $\prehilb{H}_\omega$ are
$\mathbb{Z}\times\mathbb{Z}_2$-graded and that the GNS representation
is compatible with the degrees. Even more, vectors in $\prehilb{H}_\omega$
with different degrees are orthogonal.

Since the $\mathbb{Z}$-grading of the BRST algebra is induced by $\gamma$,
we require in addition $\pi_\omega(\gamma) = \gamma_{\prehilb{H}_\omega}$.
Therefore, one needs a further condition on $\omega$ as a straightforward computation shows:

\begin{proposition}
  Consider a BRST *-algebra $(\mathcal{A},\gamma,\Theta,^*)$ with
	imaginary ghost operator and an even, positive linear functional
	$\omega\colon \mathcal{A}\rightarrow \ring{C}$.
	If the ghost charge satisfies $\gamma\in \mathcal{I}_\omega$, i.e.
	$\omega(a\gamma)=\omega(\gamma^* a)=0$ for all $a\in \mathcal{A}$,
	then $\omega$ is homogeneous with respect to the
	$\mathbb{Z}\times\mathbb{Z}_2$-grading and one has
  \begin{equation}
    \pi_\omega(\gamma)
		=
		\gamma_{\prehilb{H}_\omega}.
  \end{equation}
  In particular, $(\prehilb{H}_\omega = \mathcal{A}/\mathcal{I}_\omega,
  \gamma_\omega = \pi_\omega(\gamma), \Theta_\omega = \pi_\omega(\Theta))$
  is a BRST pre-Hilbert space and
  \begin{equation}
    \pi_\omega \colon
		\mathcal{A}
		\longrightarrow
		\Bounded^{(\bullet)}(\prehilb{H}_\omega)
  \end{equation}
  is a BRST *-representation of $\mathcal{A}$.
\end{proposition}

Thus we have found a way to generalize the GNS construction
to BRST *-algebras with imaginary ghost operator, which gives us
an explicit method to construct BRST *-representations.

\begin{remark}
The next question one could ask is if for such a linear functional
$\omega$ the reduced representations $\widetilde{(\pi_\omega)}_\BRST$ and
$\widetilde{(\pi_\omega)}_\red$ are again GNS representations
of some linear functionals on the quotients. In particular, it is
interesting if there is a canonical way to construct the corresponding
linear functionals if they exist. It turns out that there is a positive
answer to both questions if one
requires in addition $\omega(\Delta)= 0$. This reflects the compatibility
of $\omega$ with the BRST charge $\Theta$ and its adjoint $\Theta^*$ that
are responsible for the reduction.
\end{remark}

\begin{remark}
Note that *-involutions and analogues of GNS representations are 
also studied in the more general context of involutive categories,
see e.g. \cite{jacobs:2012a}. One can interpret the BRST *-algebras as 
special involutive monoids in the involutive monoidal category of mixed 
complexes: Recall that the objects $(K^\bullet,d,D)$ in this category are 
$\mathbb{Z}$-graded $\ring{C}$-modules $K^\bullet$ that are both chain complex 
$(K^\bullet,d)$ and cochain complex $(K^\bullet,D)$, where $dD + Dd$ may be non-zero. 
One can check that the functor $K^\bullet \mapsto \cc{K}^\bullet$ given by 
$\cc{K}^n = \cc{K^{-n}}$ turns the category into an involutive monoidal category. 
BRST *-algebras are then special involutive 
monoids, where the chain and cochain map as well as the grading are given by 
inner derivations. The above GNS construction turns out to be a special case of
\cite[Thm.~7.1]{jacobs:2012a}. However, we want to stress that in our setting 
positivity plays a crucial role since we are interested in *-representations on 
pre-Hilbert spaces. In the general framework we are not aware of such an 
implementation of positivity.
\end{remark}

%
%
\section{*-Involutions for the Quantum BRST Algebra}
\label{section:InvolutionsforQuantumBRSTAlgebra}

In this section we apply the above results and construct a
*-involution for the quantum BRST algebra
$\mathcal{A}^{(\bullet)}[[\lambda]] = \Anti^\bullet \liealg{g}^*
\otimes \Anti^\bullet \liealg{g} \otimes \Cinfty(M)[[\lambda]]$
corresponding to a Hamiltonian quantum $\group{G}$-space
$(M,\star,\group{G},\boldsymbol{J})$.

%
%
\subsection{Graded *-Involutions on the Grassmann Algebra}
\label{subsection:GradedInvolutiononGrassmann}

Let us assume that we have an equivariant and Hermitian star product
on $M$, so that the complex conjugation is an involution on
$\Cinfty(M)[[\lambda]]$.  Thus we only need to find a suitable
*-involution for the Grassmann algebra leading to a quantum BRST
algebra having sufficiently many positive functionals.

A first possibility for an involution is the complex conjugation.
Unfortunately, we can check that it is neither a graded nor a super
*-involution with respect to the standard ordered star product.

We define a standard ordered representation in analogy with the case
of cotangent bundles \cite{neumaier:1998a}.  Let $\iota^*$ be the
restriction
\begin{equation}
  \iota^*\colon
  \Anti^\bullet(\liealg{g}_\mathbb{C}^*\oplus
  \liealg{g}_\mathbb{C})[[\lambda]]
  \longrightarrow
  \Anti^\bullet \liealg{g}_\mathbb{C}^*[[\lambda]],
\end{equation}
i.e. $\iota^*$ sets all forms with a nontrivial $\liealg{g}$-part to
zero.  Moreover, we denote by
\begin{equation}
  \pr^* \colon
  \Anti^\bullet \liealg{g}_\mathbb{C}^*[[\lambda]]
  \longrightarrow
  \Anti^\bullet (\liealg{g}_\mathbb{C}^*\oplus
  \liealg{g}_\mathbb{C})[[\lambda]]
\end{equation}
the inclusion map. It immediately follows $\iota^*\pr^* =
\id_{\Anti^\bullet \liealg{g}_\mathbb{C}^*[[\lambda]]}$ and we can
define the following representation.

\begin{definition}[Standard ordered representation]
   \label{defi:StandardOrderedRepGrassmann}
   The \emph{standard ordered representation}
   \begin{equation}
     \rho_\std \colon
     \left( \Anti^\bullet
	     (\liealg{g}_\mathbb{C}^*\oplus \liealg{g}_\mathbb{C})[[\lambda]],
	     \circ_\std
	   \right)
	   \longrightarrow
	   \End(\Anti^\bullet \liealg{g}_\mathbb{C}^*[[\lambda]])
   \end{equation}
   of $\left(\Anti^\bullet (\liealg{g}_\mathbb{C}^*\oplus
   \liealg{g}_\mathbb{C})[[\lambda]], \circ_\std\right)$ is defined by
   \begin{equation}
     \label{eq:RhoStdDefinition}
     \rho_\std(a)\alpha
     =
     \iota^*\left(a \circ_\std \pr^*\alpha \right)
   \end{equation}
   for $a\in \Anti^\bullet(\liealg{g}_\mathbb{C}^*\oplus
   \liealg{g}_\mathbb{C})[[\lambda]]$ and $\alpha\in \Anti^\bullet
   \liealg{g}_\mathbb{C}[[\lambda]]$.
\end{definition}
Then we directly see that $\rho_\std$ is
$\mathbb{C}[[\lambda]]$-linear and satisfies $\rho_\std(1) =
\id_{\Anti^\bullet \liealg{g}_\mathbb{C}^*[[\lambda]]}$.

\begin{remark}
  The idea comes from the theory of Clifford algebras, from which we
  know
  \begin{equation}
    \label{eq:CliffordIso}
    \Anti^\bullet\left(
    \liealg{g}_\mathbb{C}^*\oplus \liealg{g}_\mathbb{C}
    \right)
    \cong
    \Clifford(
    \liealg{g}_\mathbb{C}^* \oplus \liealg{g}_\mathbb{C};
		\SP{\argument,\argument}
    )
    \cong
    \Clifford(\liealg{g}_\mathbb{C}^* \oplus \liealg{g}_\mathbb{C})
    \cong
  \End(\Anti^\bullet \liealg{g}_\mathbb{C}^*),
\end{equation}
since all non-degenerate bilinear symmetric inner products are
equivalent on $\mathbb{C}^{2n}$, see
e.g. \cite[Prop.~2.4]{meinrenken:2013a}. Note that here the first isomorphism
is an isomorphism of vector spaces, whereas the other two are isomorphisms
of Clifford algebras. We transferred this idea to
the quantized setting.
\end{remark}
\begin{proposition}
	\label{prop:StandardOrderedRep}
	The standard ordered representation $\rho_\std$ defined by
	\eqref{eq:RhoStdDefinition} is a faithful representation of
	$\left(\Anti^\bullet (\liealg{g}_\mathbb{C}^*\oplus
	\liealg{g}_\mathbb{C})[[\lambda]],\circ_\std\right)$ on
	$\Anti^\bullet \liealg{g}_\mathbb{C}^*[[\lambda]]$, i.e. we
	have
	\begin{equation}
	  \rho_\std(a) \rho_\std(\widetilde{a})\alpha
	  =
	  \rho_\std ( a \circ_\std \widetilde{a}) \alpha.
	\end{equation}
	for all $a,\widetilde{a} \in \Anti^\bullet
	(\liealg{g}_\mathbb{C}^*\oplus \liealg{g}_\mathbb{C})[[\lambda]]$
	and $\alpha \in \Anti^\bullet \liealg{g}_\mathbb{C}^*[[\lambda]]$.
\end{proposition}
\begin{proof}
  The properties follow from lengthy but straightforward computations.
\end{proof}
Using the definition of $\rho_\std$ we can immediately compute
\begin{equation}
  \label{eq:RhoStdofBasis}
  \rho_\std(1)
  =
  \id,
  \quad
  \rho_\std(\basis{e}_i)
  =
  2 \I\lambda \ins(\basis{e}_i)
  \quad
  \text{and}
  \quad
  \rho_\std(\basis{e}^i)
  =
  \basis{e}^i \wedge \cdot \;
\end{equation}
for the elements $1,\basis{e}_1,\dots,\basis{e}_n,\basis{e}^1,\dots,
\basis{e}^n$ that generate $(\Anti^\bullet(\liealg{g}_\mathbb{C}^*\oplus
\liealg{g}_\mathbb{C})[[\lambda]],\circ_\std)$.

It is known that $\Anti^\bullet \liealg{g}_\mathbb{C}^*$ has a structure
of a pre-Hilbert space over $\mathbb{C}$, which extends to
$\mathbb{C}[[\lambda]]$.

\begin{lemma}
  Let $g$ be a positive definite symmetric bilinear
  inner product on $\liealg{g}$. Then it induces a positive definite sesquilinear
  product $\SP{\argument,\argument}_*$ on $\Anti^\bullet
	\liealg{g}_\mathbb{C}^*[[\lambda]]$   via
  \begin{equation}
    \label{eq:ProductonGrassmann}
	  \SP{a_1\wedge \dots\wedge a_k, b_1\wedge \dots \wedge b_k}_*
	  =
	  \mathsf{det}(g^{-1}(a_i,b_j))
  \end{equation}
  for all $a_1,\dots,a_k,b_1,\dots,b_k \in \liealg{g}_\mathbb{C}^*$.
  In particular, $(\Anti^\bullet \liealg{g}_\mathbb{C}^*[[\lambda]],
	\SP{\argument,\argument}_*)$ is a pre-Hilbert space over
	$\mathbb{C}[[\lambda]]$.
\end{lemma}
In order to get an involution on
$\Anti^\bullet(\liealg{g}_\mathbb{C}^*\oplus
\liealg{g}_\mathbb{C})[[\lambda]]$ which is independent of $\lambda$,
we define the rescaled inner product $\SP{\argument,\argument}$ for each
$\Anti^k \liealg{g}_\mathbb{C}^*[[\lambda]]$ by
\begin{equation}
  \label{eq:ProductonGrassmannScaled}
  \SP{a,b}
  =
  (2\lambda)^k \SP{a,b}_* ,
  \quad
  \text{where}
  \quad
  a,b \in \Anti^k \liealg{g}_\mathbb{C}^*[[\lambda]].
\end{equation}
Note that for $\lambda=0$, the inner product $\SP{\argument,\argument}$ on
$\Anti^\bullet \liealg{g}_\mathbb{C}^*[[\lambda]]$ is degenerate, but
the corresponding *-involution on
$\Anti^\bullet(\liealg{g}_\mathbb{C}^*\oplus
\liealg{g}_\mathbb{C})[[\lambda]]$
resp. $\Anti^\bullet(\liealg{g}_\mathbb{C}^*\oplus
\liealg{g}_\mathbb{C})$ is still well-defined.

\begin{proposition}
  \label{prop:GrassmannStarRep}
  The standard ordered representation $\rho_\std$ of
  $\Anti^\bullet(\liealg{g}_\mathbb{C}^*\oplus \liealg{g}_\mathbb{C})[[\lambda]]$
  from Definition~\ref{defi:StandardOrderedRepGrassmann} is
  a *-representation with respect to the graded *-involution induced by
  \begin{equation}
    \label{eq:StarInvolutionGrassmann}
    \xi^*
    =
    -\I g^\flat(\xi)
    \quad
    \text{and}
    \quad
    \alpha^*
    =
    -\I g^\sharp(\alpha)
    \quad
    \forall \; \xi \in \liealg{g}_\mathbb{C},
    \; \alpha\in \liealg{g}^*_\mathbb{C}.
  \end{equation}
  Moreover, $\gamma=\frac{1}{2}\basis{e}^k \wedge \basis{e}_k$ fulfils
  $\gamma^* = - \gamma$.
\end{proposition}
\begin{proof}
  For $c\in\mathbb{C}$ we have
  \begin{align*}
    \SP{\rho_\std(\basis{e}_\ell)^*c,\basis{e}^j}
    =
    \SP{c,\rho_\std(\basis{e}_\ell) \basis{e}^j}
    =
    2\I\lambda  \SP{c, \ins(\basis{e}_\ell) \basis{e}^j}
    =
    2\I\lambda \cc{c} \delta^j_\ell
    =
    \I g_{\ell k}\cc{c}\SP{ \basis{e}^k,\basis{e}^j}
    =
    \SP{\rho_\std(-\I g_{\ell k}\basis{e}^k)c ,\basis{e}^j},
\end{align*}
and analogously $\SP{\rho_\std(\basis{e}^\ell)^* \basis{e}^j , c} =
\SP{\rho_\std(\frac{g^{\ell k}}{\I}\basis{e}_k)\basis{e}^j,c}$.
In other words, we get
$
  \rho_\std(\basis{e}_\ell)^*
  =
  \rho_\std(-\I g^\flat(\basis{e}_\ell))
$
 and
 $
  \rho_\std(\basis{e}^\ell)^*
  =
  \rho_\std(-\I g^\sharp(\basis{e}^\ell))
$.
Furthermore we have $\rho_\std(c)^*=\rho_\std(\cc{c})$,
inducing the involution from \eqref{eq:StarInvolutionGrassmann}.
Finally we compute
\begin{align*}
  2 \gamma^*
  =
  (\basis{e}^k \circ_\std \basis{e}_k)^*
  =
  - g^{km}g_{kn} \basis{e}^n \circ_\std \basis{e}_m
  =
  - 2\gamma.
\end{align*}
\end{proof}

The above result shows that this graded *-involution is in some sense
a ``natural'' one since it is induced by the above representation
$\rho_\std$.  Moreover, one can show that the standard ordered
representation $\rho_\std$ from
Definition~\ref{defi:StandardOrderedRepGrassmann} is even unitarily
equivalent to a GNS representation.

Finally, we want to show that we have sufficiently many positive
linear functionals. We recall \cite[Def.~2.7]{bursztyn.waldmann:2001a}:
Let $(\algebra{A},^*)$ be a *-algebra over $\ring{C}= \ring{R}(\I)$. Then
$\algebra{A}$ has \emph{sufficiently many positive linear
functionals} if for any non-zero Hermitian element
$h=h^*\in \algebra{A}\setminus \{0\}$	there exists a positive
linear functional $\omega \colon \algebra{A} \longrightarrow
\ring{C}$ with $\omega(h) \neq 0$.

The non-deformed Grassmann algebra has obviously not sufficiently many
positive linear functionals as $a^*\wedge a =0$ for all $a\in \Anti^k
\liealg{g}^* \otimes \Anti^\ell \liealg{g}$ with $k+\ell > n$, hence the
Cauchy-Schwarz inequality implies $\omega(a)=0$ for all such $a$ and
all positive linear functionals $\omega$, in particular for the
Hermitian ones.

\begin{proposition}
	\label{prop:DeformedGrassmannSufficientlyPosFunc}
	Let $g$ be a positive definite and symmetric bilinear
	inner product, 	inducing the involution $^*$ via
	\eqref{eq:StarInvolutionGrassmann}. Then the deformed Grassmann algebra
	$(\Anti^\bullet(\liealg{g}_\mathbb{C}^*\oplus \liealg{g}_\mathbb{C})[[\lambda]],
	\circ_\std,^*)$ has sufficiently many positive linear functionals.
\end{proposition}
\begin{proof}
  At first, note that if $\omega \colon
  \Anti^\bullet(\liealg{g}_\mathbb{C}^*\oplus
  \liealg{g}_\mathbb{C})[[\lambda]] \longrightarrow
  \mathbb{C}[[\lambda]]$ is a positive linear functional, then
  \begin{align*}
    \omega_b \colon
    \Anti^\bullet(
    \liealg{g}_\mathbb{C}^*\oplus \liealg{g}_\mathbb{C}
    )[[\lambda]]  \ni
    a
    \longmapsto
    \omega_b(a)=\omega(b^* \circ_\std a \circ_\std b)
    \in \mathbb{C}[[\lambda]]
\end{align*}
  is also positive and linear.
Consider now the projection $\delta\colon \Anti^\bullet(
\liealg{g}_\mathbb{C}^*\oplus \liealg{g}_\mathbb{C})[[\lambda]]
\longrightarrow \mathbb{C}[[\lambda]]$ that is a positive linear
functional.  For elements $c \in \mathbb{R}[[\lambda]]\setminus \{0\}$
we have $\delta(c)\neq 0$, hence we can restrict ourselves to elements
of non-trivial $\mathbb{Z}\times \mathbb{Z}$-degree. Take an
orthonormal basis $\basis{e}_1, \dots, \basis{e}_n$ of $\liealg{g}$
with respect to $g$ with orthonormal dual basis $\basis{e}^1, \dots,
\basis{e}^n$ of $\liealg{g}^*$ with respect to $g^{-1}$. In
particular, we have $\basis{e}_j^* = -\I \basis{e}^j$. Then every
non-zero Hermitian element $h$ has to be the sum of elements of the
form $ a = (-\I )^{i+j} \;\cc{c} \; \basis{e}^{k_1}\wedge \dots \wedge
\basis{e}^{k_j}\wedge \basis{e}_{\ell_i}\wedge \dots \wedge
\basis{e}_{\ell_1} + c \; \basis{e}^{\ell_1}\wedge \dots \wedge
\basis{e}^{\ell_i}\wedge \basis{e}_{k_j} \wedge \dots \wedge
\basis{e}_{k_1} $ with $i,j=0,\dots,n$ not both equal to zero and with
$c\in \mathbb{C}[[\lambda]]\setminus \{0\}$. Choose now $ b = c_1 \;
\basis{e}^{k_1}\wedge \dots \wedge \basis{e}^{k_j} + c_2 \;
\basis{e}^{\ell_1}\wedge \dots \wedge \basis{e}^{\ell_i} $ with $c_1,c_2 \in
\mathbb{C}[[\lambda]]\setminus \{0\}$.
Since $a \circ_\std b = \mu \circ e^{2\I\lambda
\jns(\basis{e}^k)\otimes \ins(\basis{e}_k)} (a\otimes b)$ we
get
\begin{align*}
  \delta_b(a)
  =
  \delta ( b^* \circ_\std a \circ_\std b)
  & =
  (2\I\lambda)^{i+j} (-\I)^i
  \left( (-1)^j \; \cc{c} \cc{c_1}c_2 +  c c_1 \cc{c_2}\right).
\end{align*}
Choosing for example $c_1 = \cc{c}$ as well as $c_2$ such that
$(-1)^j c_2 = \cc{c_2}$ yields $\delta_b(a) \neq 0$.
The above procedure can be easily extended to a general
Hermitian element.
\end{proof}

\begin{remark}
  Even though the complex conjugation yields no involution for the standard
	ordered star product, one can check that it is a well-defined super
	*-involution for the Weyl ordered one, see
	\cite{bordemann.herbig.waldmann:2000a} for a definition. However, in
	this setting one can show that there are no non-trivial positive
	linear functionals.
\end{remark}

%
%
\subsection{Comparison of the Reduced Quantum BRST Algebras}
\label{subsection:ComparisonofReducedQuantumBRSTAlgebras}

Throughout this section assume that
$(M,\star,\group{G},\boldsymbol{J}, C)$ is a Hamiltonian quantum
$\group{G}$-space with regular constraint surface, proper action on
$M$ and Hermitian star product.  Moreover, choose a positive definite
inner product $g$ on the Lie algebra $\liealg{g}$ with induced
involution $^*$ on $\left(\Anti^\bullet
(\liealg{g}^*_\mathbb{C}\oplus\liealg{g}_\mathbb{C})[[\lambda]],
\circ_\std\right)$ as in \eqref{eq:StarInvolutionGrassmann}.

\begin{lemma}
	\label{lemma:QuantumBRSTStarAlgebrawithImaginaryGhost}
	The triple $(\mathcal{A}[[\lambda]],\star_\std,^*)$ is a BRST
	*-algebra with imaginary ghost operator and it has sufficiently
	many positive linear functionals.
\end{lemma}
\begin{proof}
  We immediately get a graded *-involution on the whole quantum BRST
  algebra $\mathcal{A}^{(\bullet)}[[\lambda]]$, again denoted by $^*$,
  via $(\alpha\otimes f)^* = \alpha^* \otimes \cc{f}$.  In particular,
  it follows that
  \begin{equation*}
    \left( (\alpha\otimes f)\star_\std (\beta \otimes g)\right)^*
    =
    (\alpha \circ_\std \beta)^* \otimes \cc{f\star g}
    =
    (\beta^* \circ _\std\alpha^*) \otimes (\cc{g}\star \cc{f}) )
    =
    (\beta \otimes g)^* \star_\std (\alpha \otimes f)^*
  \end{equation*}
  for all $\alpha,\beta \in \Anti^\bullet(\liealg{g}^*_\mathbb{C}
  \oplus \liealg{g}_\mathbb{C})[[\lambda]]$ and $f,g\in
  \Cinfty(M)[[\lambda]]$.  For the BRST charge $\gamma = \frac{1}{2}
  \basis{e}^k \wedge \basis{e}_k$ we have already seen $\gamma^* =
  -\gamma$ in Proposition~\ref{prop:GrassmannStarRep}, thus the ghost number
  derivation $\Gh = \frac{1}{\I\lambda}\ad(\gamma)$ fulfils $ \Gh^*
  = - \Gh.  $ Therefore, we have constructed a graded *-involution
  with imaginary ghost operator.  The only thing remaining to be shown
  is that the quantum BRST algebra has sufficiently many positive
  linear functionals. By
  Proposition~\ref{prop:DeformedGrassmannSufficientlyPosFunc} the Grassmann
  part has sufficiently many positive linear functionals and
  $(\Cinfty(M)[[\lambda]],\star)$ with complex conjugation as
  involution has sufficiently many positive linear functionals, see
  \cite[Prop.~5.3]{bursztyn.waldmann:2000a}. Moreover,
  \cite[Prop.~2.8]{bursztyn.waldmann:2001a} states that a unital
  *-algebra has sufficiently many positive linear functionals if and
  only if it has a faithful *-representation on a pre-Hilbert
  space. Both the Grassmann algebra
  $\left(\Anti^\bullet(\liealg{g}^*_\mathbb{C}\oplus
  \liealg{g}_\mathbb{C})[[\lambda]],\circ_\std\right)$ and the
  functions are unital *-algebras and the Grassmann part has already
  such a *-representation $\rho_\std $
  as discussed in Proposition~\ref{prop:GrassmannStarRep}.  In particular,
  we know
  \begin{equation}
    \label{eq:ProofGrassmannRep}
    \tag{$*$}
  \rho_\std( x_{ij}) \colon
  \Anti^k \liealg{g}_\mathbb{C}^*[[\lambda]]
  \longrightarrow
  \begin{cases}
    \Anti^{k + i - j} \liealg{g}_\mathbb{C}^*[[\lambda]]
    \quad&\text{for}\; k \geq j   \\
    \{0\} &\text{else}
  \end{cases}
  \end{equation}
  for all $x_{ij} \in \Anti^i \liealg{g}_\mathbb{C}^*[[\lambda]]
  \otimes \Anti^j \liealg{g}_\mathbb{C}[[\lambda]]$.  Let $\pi \colon
  \Cinfty(M)[[\lambda]] \longrightarrow \prehilb{H}$ be a faithful
  *-representation of the functions on a pre-Hilbert space
  $\prehilb{H}$. It remains to show that
  \begin{equation*}
    \rho
    =
    \rho_\std \otimes \pi \colon
    \mathcal{A}[[\lambda]]
    \longrightarrow
    \Bounded(\Anti^\bullet \liealg{g}_\mathbb{C}^*[[\lambda]]
    \otimes \prehilb{H})
  \end{equation*}
  is injective. Consider a general element $\sum_{i,j,\alpha_{ij}}
  x_{ij \alpha_{ij}} \otimes F^{ij \alpha_{ij}} \in
  \mathcal{A}[[\lambda]]$, where $\{x_{ij\alpha_{ij}}
  \}_{\alpha_{ij}}$ is a basis of $\Anti^i
  \liealg{g}_\mathbb{C}^*\otimes \Anti^j \liealg{g}_\mathbb{C}$, and
  where $F^{ij\alpha_{ij}}\in \Cinfty(M)[[\lambda]]$.
Let $k$ be the minimal index such that $\sum_{i\alpha_{i k}}
x_{i k \alpha_{ik}}\otimes F^{i k \alpha_{ik}} \neq 0$. Using
 \eqref{eq:ProofGrassmannRep} we have
$\rho_\std(x_{ik\alpha_{ik}}) z \in \Anti^i \liealg{g}^*_\mathbb{C}$
for $z \in \Anti^k\liealg{g}^*_\mathbb{C}[[\lambda]]$,
so the images for different $i=0,1,\dots,n$ are either zero or
linearly independent, allowing us to fix the index $i$.
A straightforward computation shows
\begin{equation*}
  \rho_\std (
    \basis{e}^{j_1}\wedge \dots \wedge \basis{e}^{j_i} \wedge
    \basis{e}_{\ell_k}\wedge \dots \wedge\basis{e}_{\ell_1}
  ) z
  =
  (2\I\lambda)^k \basis{e}^{j_1} \wedge \dots \wedge \basis{e}^{j_i}
  \wedge \ins(\basis{e}_{\ell_k})\wedge \dots \wedge\ins(\basis{e}_{\ell_1})
  z.
\end{equation*}
Choosing now $\phi \in \prehilb{H}$ such that
$\pi\left( F^{r_k \dots r_1}_{p_1\dots p_i}\right) \phi \neq 0$ for
some sets of indices $\{r_1,\dots,r_k\}$ and $\{p_1,\dots,p_i\}$ yields
\begin{align*}
  \rho &\left(
    \basis{e}^{j_1}\wedge \dots \wedge \basis{e}^{j_i} \wedge
    \basis{e}_{\ell_k}\wedge \dots \wedge\basis{e}_{\ell_1}  \otimes
    F^{\ell_k \dots \ell_1}_{j_1\dots j_i}
  \right)\left(
    \basis{e}^{r_1}\wedge \cdots \wedge \basis{e}^{r_k} \otimes \phi
  \right)                                                           \\
  & =
  k! \; (2\I\lambda)^k \basis{e}^{j_1}\wedge \dots \wedge
  \basis{e}^{j_i}\otimes\pi\left(F^{r_k\dots r_1}_{j_1\dots j_i}\right) \phi
  \neq
  0.
\end{align*}
\end{proof}
Now we can define as in the general setting from
Section~\ref{section:AbstractBRSTalgebras} an \emph{adjoint standard
  ordered BRST operator} $\boldsymbol{D}_\std^*$ by
\begin{equation}
  \label{eq:AdjointStandradOrderedBRSTOperator}
  \boldsymbol{D}_\std^*
  =
  \frac{1}{\I \lambda} \ad_\std ( \Theta_\std^*).
\end{equation}
We get two different quantum BRST cohomologies:
on one hand, the usual quantum BRST cohomology
$\boldHBRST^{(\bullet)}(\mathcal{A}[[\lambda]])$ from
Proposition~\ref{prop:QuantumBRSTCohomology}
with corresponding reduced quantum BRST algebra
\begin{equation}
  \mathcal{A}_\red
  =
  \boldHBRST^{(0)}(\mathcal{A}[[\lambda]])
  =
  \frac{\ker\boldsymbol{D}_\std \cap \mathcal{A}^{(0)}[[\lambda]] }
       {\image \boldsymbol{D}_\std \cap \mathcal{A}^{(0)}[[\lambda]]}.
\end{equation}
On the other hand, we have the BRST quotient
$\boldHBRSTtilde^{(\bullet)} (\mathcal{A}[[\lambda]])$
from \eqref{eq:BRSTQuotient}
with corresponding \emph{reduced quantum BRST *-algebra}
\begin{equation}
  \widetilde{\mathcal{A}}_\red
  =
  \boldHBRSTtilde^{(0)}(\mathcal{A}[[\lambda]])
  =
  \frac{\ker\boldsymbol{D}_\std \cap \ker\boldsymbol{D}_\std^* \cap
        \mathcal{A}^{(0)}[[\lambda]] }
   {\image \boldsymbol{D}_\std \cap\image \boldsymbol{D}_\std^*\cap
		\mathcal{A}^{(0)}[[\lambda]]}
\end{equation}
that is indeed a *-algebra by
Lemma~\ref{lemma:ReducedBRSTStarAlgebra}.  Therefore, the natural
question is whether we can compare
$\boldHBRST^{(\bullet)}(\mathcal{A}[[\lambda]])$ with
$\boldHBRSTtilde^{(\bullet)}(\mathcal{A}[[\lambda]])$, in particular
in ghost number zero. The rest of the section consists in the proof of
the following main result.

\begin{theorem}
	\label{thm:AredtildeIsomorphicAred}
	Let $(M,\star,\group{G},\boldsymbol{J},C)$ be a Hamiltonian quantum
	$\group{G}$-space with regular constraint surface, compact
	Lie group and Hermitian star product $\star$. Moreover, choose a
	positive definite inner product on the corresponding Lie algebra
	$\liealg{g}$, inducing the involution $^*$. Then one has
	\begin{equation}
    \label{eq:IsomorphismReducedBRSTAlg}
    \widetilde{\mathcal{A}}_\red
	  \cong
	  \Cinfty(C)^\group{G}[[\lambda]]
	  \cong
	  \mathcal{A}_\red
  \end{equation}
	with isomorphism $\boldsymbol{\iota^*}\colon
	\widetilde{\mathcal{A}}_\red \longrightarrow
	\Cinfty(C)^\group{G}[[\lambda]]$ and inverse $\prol$.
\end{theorem}

We already know from Proposition~\ref{prop:QuantumBRSTCohomology} that there
is an isomorphism
\begin{equation*}
  \boldsymbol{\iota^*} \colon
  \mathcal{A}_\red
  =
  \frac{\ker \boldsymbol{D}_\std \cap \mathcal{A}^{(0)}[[\lambda]]}
        {\image \boldsymbol{D}_\std \cap \mathcal{A}^{(0)}[[\lambda]] }
  \longrightarrow
  \Cinfty(C)^\group{G}[[\lambda]]
\end{equation*}
with inverse $\widehat{\boldsymbol{h}}\at{\Cinfty(C)^\group{G}[[\lambda]]}
=\prol$, where we understand $\boldsymbol{\iota^*}$ to act on the
representatives of the equivalence classes and $\prol$ to map into
the corresponding equivalence class. We directly observe the following:

\begin{lemma}
  The map
  \begin{equation}
    \label{eq:IotaonReducedBRSTStarAlg}
    \boldsymbol{\iota^*}  \colon
 	  \widetilde{\mathcal{A}}_\red
    \longrightarrow
    \Cinfty(C)^\group{G}[[\lambda]]
  \end{equation}
  is well-defined with right inverse $\prol$. Moreover,
	$\prol$ is also a left inverse if
	\begin{equation}
    \label{eq:DifferenceInIntersectionImages}
    \sum_{i,\alpha_i } \prol\boldsymbol{\iota^*}
    \left( x_{i\alpha_i }\otimes F^{i\alpha_i } \right)
    -
    \sum_{i,\alpha_i } x_{i\alpha_i }\otimes F^{i\alpha_i }
    \in
    \image \boldsymbol{D}_\std \cap \image \boldsymbol{D}_\std^*
  \end{equation}
	holds for any element $\sum_{i,\alpha_i } x_{i\alpha_i }\otimes F^{i\alpha_i }
  \in\ker \boldsymbol{D}_\std \cap \ker \boldsymbol{D}_\std^*
  \cap \mathcal{A}^{(0)}[[\lambda]]$, where
  $\{x_{i\alpha_i }\}_{\alpha_i }$ denotes a basis of $\Anti^i\liealg{g}^*
  \otimes \Anti^i \liealg{g}$, and
  $F^{i\alpha_i }\in \Cinfty(M)[[\lambda]]$ for any $i=0,1,\dots,n$
  and $\alpha_i $.
\end{lemma}
\begin{proof}
  The map~\eqref{eq:IotaonReducedBRSTStarAlg} is well-defined as
  $\boldsymbol{\iota^*}$ vanishes on $ \image \boldsymbol{D}_\std \cap
  \image \boldsymbol{D}^*_\std \cap \mathcal{A}^{(0)}[[\lambda]]$.
  Moreover, since $\phi \in \Cinfty(C)^\group{G}[[\lambda]]$ implies
  $\cc{\phi}\in \Cinfty(C)^\group{G}[[\lambda]]$ and since
  $\boldsymbol{\delta}=\delta$, we have $\boldsymbol{\iota^*}\prol
  =\id_{\Cinfty(C)[[\lambda]]}$ and
  \begin{align*}
    \boldsymbol{D}_\std (\prol (\phi))
    & =
    (\boldsymbol{\delta}+2\boldsymbol{\del})(\prol (\phi))
    =
    \prol \delta^c \phi
    =
    0,                       \\
    \boldsymbol{D}_\std^* (\prol (\phi ))
    & =
    \left((\boldsymbol{\delta}+2\boldsymbol{\del})
    \left(\prol\left( \cc{\phi}\right)\right)\right)^*
    =
    \left(\prol \delta^c \cc{\phi}\right)^* = 0.
  \end{align*}
  So $\prol$ is still a well-defined right inverse of
  $\boldsymbol{\iota^*}$.
\end{proof}

These conditions can be further simplified by exploiting the chain
homotopy from Proposition~\ref{prop:QuantumBRSTCohomology}.

\begin{proposition}
	\label{prop:ConditionsforRedAlgIso}
	Let $(M,\star,\group{G},\boldsymbol{J},C)$ be a Hamiltonian quantum
	$\group{G}$-space with regular constraint surface, proper action on
	$M$ and Hermitian star product $\star$. Moreover, let $g$ be a
	positive definite inner product on $\liealg{g}$ inducing the
	involution $^*$ via \eqref{eq:StarInvolutionGrassmann}.
	Then $\boldsymbol{\iota^*} \colon \widetilde{\mathcal{A}}_\red
	\longrightarrow \Cinfty(C)^\group{G}[[\lambda]]$ is an isomorphism
	with inverse $\prol$ if
	\begin{equation}
	  \label{eq:IotaCommutesWithCConFzero}
 	  \cc{\boldsymbol{\iota^*}F^0}
	  =
	  \boldsymbol{\iota^*}\cc{F^0}
	\end{equation}
	for $\sum_{i,\alpha_i } x_{i\alpha_i }\otimes F^{i\alpha_i }
	\in\ker \boldsymbol{D}_\std \cap \ker \boldsymbol{D}_\std^*
	\cap \mathcal{A}^{(0)}[[\lambda]]$
	with $F^0 = x_{0\alpha_0} \otimes
	F^{0\alpha_0} \in \Cinfty(M)[[\lambda]]$ as above.
\end{proposition}
\begin{proof}
  We consider at first the augmented standard ordered BRST operator
  $\widehat{\boldsymbol{D}}_\std = \boldsymbol{D}_\std +
  \boldsymbol{\del}^c + 2 \boldsymbol{\iota^*}$.  We know from
  Proposition~\ref{prop:QuantumBRSTCohomology} that
  $  \widehat{\boldsymbol{D}}_\std\widehat{\boldsymbol{h}}
    +   \widehat{\boldsymbol{h}}\widehat{\boldsymbol{D}}_\std  =  2\id$,
  which entails
  \begin{align*}
    \sum_{i,\alpha_i }\prol \boldsymbol{\iota^*}
    \left( x_{i\alpha_i }\otimes F^{i\alpha_i }\right)
    =
    \frac{1}{2}\sum_{i,\alpha_i }\widehat{\boldsymbol{h}}
    \widehat{\boldsymbol{D}}_\std \left(
    x_{i\alpha_i }\otimes F^{i\alpha_i }
    \right)
    =
    \sum_{i,\alpha_i } \left(
    x_{i\alpha_i }\otimes F^{i\alpha_i }
    \right)
    -
    \frac{1}{2}\sum_{i,\alpha_i }\boldsymbol{D}_\std
    \widehat{\boldsymbol{h}} \left(
    x_{i\alpha_i }\otimes F^{i\alpha_i }
    \right)
  \end{align*}
  as $\widehat{\boldsymbol{D}}_\std \widehat{\boldsymbol{h}}
  \left(\sum_{i,\alpha_i } x_{i\alpha_i }\otimes F^{i\alpha_i }\right) =
  \boldsymbol{D}_\std \widehat{\boldsymbol{h}} \left( \sum_{i,\alpha_i
  }x_{i\alpha_i }\otimes F^{i\alpha_i }\right)$.  Applying the
  *-involution yields
  \begin{align*}
    \sum_{i,\alpha_i }\cc{\prol\boldsymbol{\iota^*}
      \left( x_{i\alpha_i }\otimes F^{i\alpha_i }\right)}
    &=
    \sum_{i,\alpha_i }\left( x_{i\alpha_i }\otimes F^{i\alpha_i }\right)^*
    +
    \frac{1}{2}\sum_{i,\alpha_i }\boldsymbol{D}^*_\std
    \left(
    \widehat{\boldsymbol{h}}
    \left(
    x_{i\alpha_i }\otimes F^{i\alpha_i }
    \right)
    \right)^*
  \end{align*}
  by the definition of $\boldsymbol{D}^*_\std$.  Because of
  $\left(\sum_{i,\alpha_i }x_{i\alpha_i }\otimes F^{i\alpha_i }\right)^*
  \in\ker \boldsymbol{D}_\std \cap \ker \boldsymbol{D}_\std^* \cap
  \mathcal{A}^{(0)}[[\lambda]]$ we also get
  \begin{align*}
    \sum_{i,\alpha_i }\left( x_{i\alpha_i }\otimes F^{i\alpha_i }\right)^*
    & =
    \sum_{i,\alpha_i }\prol\boldsymbol{\iota^*}
    \left( x_{i\alpha_i }\otimes F^{i\alpha_i }\right)^*
    +
    \frac{1}{2}\sum_{i,\alpha_i }\boldsymbol{D}_\std \widehat{\boldsymbol{h}}
    \left( x_{i\alpha_i }\otimes F^{i\alpha_i }\right)^*.
  \end{align*}
  Thus to prove the desired \eqref{eq:DifferenceInIntersectionImages}
  it suffices to show
  \begin{equation*}
    \sum_{i,\alpha_i }\cc{\prol\boldsymbol{\iota^*}
      \left( x_{i\alpha_i }\otimes F^{i\alpha_i }\right)}
    -
    \sum_{i,\alpha_i }\prol\boldsymbol{\iota^*}
    \left( x_{i\alpha_i }\otimes F^{i\alpha_i }\right)^*
    \in
    \image \boldsymbol{D}_\std^*,
  \end{equation*}
  which is fulfilled if $\cc{\boldsymbol{\iota^*}F^0} =
  \boldsymbol{\iota^*}\cc{F^0}$.
\end{proof}

In general we do not know if $F^0$ is $\group{G}$-invariant, i.e.
$\boldsymbol{\delta}F^0=\delta F^0=0$, as the higher orders of
$\sum_{i, \alpha_i }x_{i\alpha_i } \otimes F^{i\alpha_i }$ could cancel
this term under $\boldsymbol{D}_\std$.  Hence we can not apply
\cite[Cor.~4.6]{gutt.waldmann:2010a}, giving exactly the property
$\boldsymbol{\iota^*}\cc{f}=\cc{\boldsymbol{\iota^*}f}$ for invariant
functions $f\in \Cinfty(M)[[\lambda]]$. The idea to check
\eqref{eq:IotaCommutesWithCConFzero} is to construct an inner product
on $\Cinfty(C)[[\lambda]]$ with values in
$\Cinfty(C)^\group{G}[[\lambda]]$ and a corresponding *-representation
of $(\Cinfty(M)[[\lambda]],\star)$, see
\cite[Def.~5.4]{gutt.waldmann:2010a}.

\begin{definition}[Algebra-valued inner product]
	\label{definition:algebravaluedinnerproduct}
	Let $\group{G}$ be a compact Lie group.
	The $\Cinfty(C)^\group{G}[[\lambda]]$-\emph{valued inner product}
	on $\Cinfty(C)[[\lambda]]$ is pointwise defined by
	\begin{equation}
	  \label{eq:SPred}
	  \SP{\phi, \psi}_\red (c)
	  =
	  \int_\group{G}
	  \left(
	    \boldsymbol{\iota^*}
		  \left(
	      \cc{\prol(\phi)} \star \prol(\psi)
	    \right)
	  \right)
	  (\Phi_{g^{-1}}(c))
	  \D^\lefttriv g
	\end{equation}
	for all $\phi, \psi \in \Cinfty(C)[[\lambda]]$ and $c\in C$, where
	$\D^\lefttriv g$ denotes the left invariant Haar measure.
\end{definition}
Then $\SP{\argument,\argument}_\red$ is well-defined,
$\mathbb{C}[[\lambda]]$-sesquilinear and can be
rewritten in the following way, see
\cite[Lemmas~5.6, 5.8 and 5.9]{gutt.waldmann:2010a}.

\begin{proposition}
	\label{prop:AlgebraValuedInnerProduct}
	The map $\SP{\argument, \argument}_\red$ defines a non-degenerate
	inner product on $\Cinfty(C)[[\lambda]]$ with values in the
	invariant functions $\Cinfty(C)^\group{G}[[\lambda]]$. One
	can rewrite it in an alternative way
	\begin{align}
	  \label{eq:SPredsimple}
	  \SP{\phi, \psi}_\red
	  =
	  \boldsymbol{\iota^*} \int_{\group{G}}
	  \Phi^*_{g^{-1}}
	  \left(
	    \cc{\prol(\phi)} \star \prol(\psi)
	  \right)
	  \D^\lefttriv g
	  =
	  \boldsymbol{\iota^*} \int_{ \group{G}}
	  \cc{\prol(\Phi^*_{g^{-1}}\phi)}
	  \star
	  \prol(\Phi^*_{g^{-1}} \psi)
	  \D^\lefttriv g.
	\end{align}
	In particular, one has for all $\phi,\psi \in \Cinfty(C)[[\lambda]]$
	\begin{equation}
 	  \label{eq:SPredHermitian}
	  \cc{\SP{\phi,\psi}}_\red
	  =
	  \SP{\psi,\phi}_\red.
	\end{equation}
\end{proposition}
Recall the left action $\bullet$ of $(\Cinfty(M)[[\lambda]],\star)$ on
$\Cinfty(C)[[\lambda]]$ from \cite[Def.~3.7]{gutt.waldmann:2010a} that
is given by
\begin{equation}
  \label{eq:LeftRepofMonC}
  \Cinfty(M)[[\lambda]] \times \Cinfty(C)[[\lambda]] \ni (f,\phi)
  \longmapsto f \bullet \phi = \boldsymbol{\iota^*} (f \star \prol
  (\phi)) \in \Cinfty(C)[[\lambda]].
\end{equation}
The key point is now that this action yields a $^*$-representation, see
\cite[Prop.~5.11]{gutt.waldmann:2010a}.

\begin{proposition}
	\label{proposition:LeftModuleIsRepresentation}
	The action $\bullet$ is a $^*$-representation of
  $(\Cinfty(M)[[\lambda]], \star)$ on $\Cinfty(C)[[\lambda]]$
  with respect to the inner product $\SP{\argument,\argument}_\red$,
  i.e. we have for all $\phi, \psi \in \Cinfty(C)[[\lambda]]$
  and $f \in \Cinfty(M)[[\lambda]]$
	\begin{equation}
	  \label{eq:fbulletIsAdjointable}
	  \SP{\phi, f \bullet \psi}_\red = \SP{\cc{f} \bullet \phi,
            \psi}_\red.
	\end{equation}
\end{proposition}

Now we can finally prove Theorem~\ref{thm:AredtildeIsomorphicAred}.
\begin{proof}[of Theorem~\ref{thm:AredtildeIsomorphicAred}]
By Proposition~\ref{prop:ConditionsforRedAlgIso}
it suffices to show
\begin{equation*}
  \cc{\boldsymbol{\iota^*}F^0}
  =
  \boldsymbol{\iota^*}\cc{F^0},
\end{equation*}
where $F^0 = 1 \otimes F^0$ is again the lowest order of some
$\sum_{i,\alpha_i } x_{i\alpha_i }\otimes F^{i\alpha_i } \in\ker
\boldsymbol{D}_\std \cap \ker \boldsymbol{D}_\std^* \cap
\mathcal{A}^{(0)}[[\lambda]]$. Just as above, $\{x_{i\alpha_i
}\}_{\alpha_i }$ is a basis of $\Anti^i\liealg{g}^* \otimes \Anti^i
\liealg{g}$ and $F^{i\alpha_i }\in \Cinfty(M)[[\lambda]]$ for all
$i=0,1,\dots,n$ and all $\alpha_i $.  By the construction of $\prol$
and assuming $M=M_\nice$ in the notation of
\cite[Sec.~2.2]{gutt.waldmann:2010a} we know $\prol(1)=1\in
\Cinfty(M)[[\lambda]]$ for the constant function
$1\in\Cinfty(C)[[\lambda]]$. Thus we get
\begin{align*}
  F^0 \bullet 1 = \boldsymbol{\iota^*}( F^0 \star 1) =
  \boldsymbol{\iota^*} F^0 \in \Cinfty(C)^\group{G}[[\lambda]]
\end{align*}
as well as $\cc{F^0} \bullet 1 = \boldsymbol{\iota^*} \cc{F^0} \in
\Cinfty(C)^\group{G}[[\lambda]]$, which implies
$\Phi^*_{g^{-1}} \boldsymbol{\iota^*}F^0=  \boldsymbol{\iota^*}F^0$
and
$\Phi^*_{g^{-1}}\boldsymbol{\iota^*}\cc{F^0}= \boldsymbol{\iota^*}\cc{F^0}$.
Using \eqref{eq:SPredsimple} we can compute
\begin{align*}
  \SP{1, F^0 \bullet 1}_\red
  =
  \int_\group{G}  \boldsymbol{\iota^*}
  \left(
    \cc{\prol\left(\Phi^*_{g^{-1}}1\right)}
    \star
    \prol\left(\Phi^*_{g^{-1}} \boldsymbol{\iota^*}F^0\right)
  \right)
  \D^\lefttriv g
  =
  \boldsymbol{\iota^*}F^0 \int_\group{G}\D^\lefttriv g,
\end{align*}
and analogously
\begin{align*}
  \SP{ \cc{F^0} \bullet 1, 1}_\red
  =
  \int_\group{G} \boldsymbol{\iota^*}
  \left(
    \cc{\prol\left(\Phi^*_{g^{-1}}\boldsymbol{\iota^*}\cc{F^0}\right)}
    \star
	\prol\left(\Phi^*_{g^{-1}} 1\right)
  \right)
  \D^\lefttriv g
  =
  \cc{\boldsymbol{\iota^*} \cc{F^0}}\int_\group{G}  \D^\lefttriv g  ,
\end{align*}
which together with \eqref{eq:fbulletIsAdjointable} implies the
desired $\boldsymbol{\iota^*}F^0= \cc{\boldsymbol{\iota^*} \cc{F^0}}$.
\end{proof}

If the action is in addition free on $C$, we know that $M_\red =
C/\group{G}$ is a smooth manifold and with
Proposition~\ref{prop:QuantumBRSTCohomology} we have an induced star
product $\star_\red$ on $\Cinfty(M_\red)[[\lambda]]\cong
\Cinfty(C)^\group{G}[[\lambda]]$ given by
\begin{equation}
  \pi^*(u_1\star_\red u_2)
  =
  \boldsymbol{\iota^*}
  (\prol(\pi^*u_1)\star\prol(\pi^* u_2))
\end{equation}
for $u_1,u_2\in\Cinfty(M_\red)[[\lambda]]$. Thus we
immediately get:

\begin{corollary}
  \label{cor:CCasInvolutiononMred}
  Let $(M,\star,\group{G},\boldsymbol{J},C)$ be a Hamiltonian quantum
  $\group{G}$-space with regular constraint surface, positive definite
  inner product on $\liealg{g}$ and Hermitian star product $\star$. In
  addition, let the compact Lie group $\group{G}$ act freely on
  $C$. Then one has
  \begin{equation}
    \label{eq:IsomorphismReducedBRSTAlgMred}
	  \widetilde{\mathcal{A}}_\red
	  \cong
	  \Cinfty(M_\red)[[\lambda]]
	  \cong
	  \mathcal{A}_\red.
  \end{equation}
  Moreover, $\boldsymbol{\iota^*}$ induces the complex conjugation
  as involution on $(\Cinfty(M_\red)[[\lambda]],\star_\red)$.
\end{corollary}
\begin{proof}
The fact that this construction induces the complex conjugation as involution
for the reduced star product follows as in \cite[Prop.~4.7]{gutt.waldmann:2010a}.
Explicitly, we have
\begin{align*}
	\pi^*\cc{(u_1\star_\red  u_2)}
	&=
	\cc{\boldsymbol{\iota^*}
	\left( \prol (\pi^* u_1)\star \prol( \pi^* u_2)\right)}
	=
	\boldsymbol{\iota^*}\left(\cc{ \prol (\pi^* u_1) \star\prol (\pi^* u_2)}\right) \\
	&=
	\boldsymbol{\iota^*}
	\left( \prol (\pi^* \cc{u_2}) \star \prol( \pi^*\cc{ u_1})\right)
	=
	\pi^*(\cc{u_2} \star_\red  \cc{u_1})
\end{align*}
for all $u_1,u_2 \in \Cinfty(M_\red)[[\lambda]]$.
\end{proof}
There exists another construction of a *-involution for $\star_\red$ via
the GNS representation for a suitably chosen positive functional depending
on a density, see \cite[Thm.~4.17]{gutt.waldmann:2010a}.
In comparison to this one we get
always the same *-involution for the reduced star product,
independently on the inner product $g$ on $\liealg{g}$. This is due to
the fact that the choice of a density for the positive functional is a
non-canonical one, whereas all inner products on the Lie algebra lead
to isomorphic reduced *-algebras. Also, in the only order where
$\boldsymbol{\iota^*}\at{\mathcal{A}^{(0)}[[\lambda]]}$ does not
vanish identically, i.e. on the functions
$\mathcal{A}^{0,0}[[\lambda]] = \Cinfty(M)[[\lambda]]$, the induced
involution is always just the complex conjugation.

\begin{remark}
	From \cite[Cor.~4.6]{gutt.waldmann:2010a} we know
	that one has
	$\boldsymbol{\iota^*}\cc{f} = \cc{\boldsymbol{\iota^*}f}$ for
	$\group{G}$-invariant functions $f\in \Cinfty(M)[[\lambda]]$, which
	implies that the complex conjugation is an involution for
	$\star_\red$.	Furthermore, as one needs this identity to show that
	$\SP{\argument,\argument}_\red$ satisfies for all
	$\phi,\psi \in \Cinfty(C)[[\lambda]]$
	\begin{align*}
	  \cc{\SP{\phi,\psi}}_\red
	  =
	  \SP{\psi,\phi}_\red
 	  \quad
	  \text{and}
	  \quad
	  \SP{\phi, f \bullet \psi}_\red
	  =
	  \SP{\cc{f} \bullet \phi, \psi}_\red,
	\end{align*}
	it is not surprising that the construction via the algebra-valued
	inner product yields again the complex conjugation as involution on
	$(\Cinfty(M_\red)[[\lambda]],\star_\red)$.
	Therefore, it is important to remark that the complex conjugation
	is induced by an isomorphism with the reduced quantum BRST
	*-algebra $\widetilde{\mathcal{A}}_\red \cong \Cinfty(M_\red)[[\lambda]]$
	if $M_\red$ exists as smooth manifold.
\end{remark}

%
%

%
%


\begin{thebibliography}{10}

\bibitem {bayen.et.al:1978a}
\textsc{Bayen, F., Flato, M., Fr{{\o}}nsdal, C., Lichnerowicz, A.,
  Sternheimer, D.:}\newblock \textit{Deformation Theory and
  Quantization}.
\newblock Ann. Phys.  \textbf{111} (1978), 61--151.
\par\csname bayen.et.al:1978achairxnote\endcsname

\bibitem {becchi.rouet.stora:1976a}
\textsc{Becchi, C., Rouet, A., Stora, R.:}\newblock
  \textit{Gauge Field Models}.
\newblock Renormalization Theory   (1976), 269--297.
\par\csname becchi.rouet.stora:1976achairxnote\endcsname

\bibitem {bordemann:2005a}
\textsc{Bordemann, M.:}\newblock \textit{(Bi)Modules,
  morphisms, and reduction of star-products: the symplectic case, foliations,
  and obstructions}.
\newblock Trav. Math.  \textbf{16} (2005), 9--40.
\par\csname bordemann:2005achairxnote\endcsname

\bibitem {bordemann.herbig.waldmann:2000a}
\textsc{Bordemann, M., Herbig, H.-C., Waldmann, S.:}\newblock
  \textit{BRST Cohomology and Phase Space Reduction in Deformation
  Quantization}.
\newblock Commun. Math. Phys.  \textbf{210} (2000), 107--144.
\par\csname bordemann.herbig.waldmann:2000achairxnote\endcsname

\bibitem {bordemann.waldmann:1998a}
\textsc{Bordemann, M., Waldmann, S.:}\newblock
  \textit{Formal GNS Construction and States in Deformation
  Quantization}.
\newblock Commun. Math. Phys.  \textbf{195} (1998), 549--583.
\par\csname bordemann.waldmann:1998achairxnote\endcsname

\bibitem {bursztyn.waldmann:2000a}
\textsc{Bursztyn, H., Waldmann, S.:}\newblock
  \textit{On Positive Deformations of {$^*$}-Algebras}.
\newblock In: \textsc{Dito, G., Sternheimer, D. (eds.):}\newblock
  \textit{Conf{\'e}rence Mosh{\'e} Flato 1999. Quantization,
  Deformations, and Symmetries}, \textit{Mathematical Physics
  Studies} no. \textbf{22},   69--80. Kluwer Academic Publishers, Dordrecht,
  Boston, London, 2000.

\bibitem {bursztyn.waldmann:2001a}
\textsc{Bursztyn, H., Waldmann, S.:}\newblock
  \textit{Algebraic Rieffel Induction, Formal Morita Equivalence
  and Applications to Deformation Quantization}.
\newblock J. Geom. Phys.  \textbf{37} (2001), 307--364.

\bibitem {bursztyn.waldmann:2005b}
\textsc{Bursztyn, H., Waldmann, S.:}\newblock
  \textit{Completely positive inner products and strong {M}orita
  equivalence}.
\newblock Pacific J. Math.  \textbf{222} (2005), 201--236.

\bibitem {bursztyn.waldmann:2005a}
\textsc{Bursztyn, H., Waldmann, S.:}\newblock
  \textit{Hermitian star products are completely positive
  deformations}.
\newblock Lett. Math. Phys.  \textbf{72} (2005), 143--152.

\bibitem {cattaneo.felder:2007a}
\textsc{Cattaneo, A.~S., Felder, G.:}\newblock
  \textit{Relative formality theorem and quantisation of
  coisotropic submanifolds}.
\newblock Adv. Math.  \textbf{208} (2007), 521--548.

\bibitem {dippell.esposito.waldmann:2019a}
\textsc{Dippell, M., Esposito, C., Waldmann, S.:}\newblock
  \textit{Coisotropic Triples, Reduction and Classical Limit}.
\newblock Documenta Mathematica  \textbf{24} (2019), 1811--1853.

\bibitem {fedosov:1996a}
\textsc{Fedosov, B.~V.:}\newblock \textit{Deformation
  Quantization and Index Theory}.
\newblock Akademie Verlag, Berlin, 1996.

\bibitem {forger.kellendonk:1992a}
\textsc{Forger, M., Kellendonk, J.:}\newblock
  \textit{Classical BRST Cohomology and Invariant Functions on
  Constraint Manifolds I}.
\newblock Commun. Math. Phys.  \textbf{143} (1992), 235--251.

\bibitem {gutt.rawnsley:2003a}
\textsc{Gutt, S., Rawnsley, J.:}\newblock
  \textit{Natural Star Products on Symplectic Manifolds and Quantum
  Moment Maps}.
\newblock Lett. Math. Phys.  \textbf{66} (2003), 123--139.

\bibitem {gutt.waldmann:2010a}
\textsc{Gutt, S., Waldmann, S.:}\newblock
  \textit{Involutions and Representations for Reduced Quantum
  Algebras}.
\newblock Adv. Math.  \textbf{224} (2010), 2583--2644.

\bibitem {henneaux.teitelboim:1992a}
\textsc{Henneaux, M., Teitelboim, C.:}\newblock
  \textit{Quantization of Gauge Systems}.
\newblock Princeton University Press, New Jersey, 1992.

\bibitem {jacobs:2012a}
\textsc{Jacobs, B.:}\newblock \textit{Involutive
  Categories and Monoids, with a GNS-Correspondence}.
\newblock Foundations of Physics  \textbf{42} (2012), 874--895.

\bibitem {kontsevich:2003a}
\textsc{Kontsevich, M.:}\newblock \textit{Deformation
  Quantization of {P}oisson manifolds}.
\newblock Lett. Math. Phys.  \textbf{66} (2003), 157--216.

\bibitem {kostant.sternberg:1987a}
\textsc{Kostant, B., Sternberg, S.:}\newblock
  \textit{Symplectic Reduction, {BRS} Cohomology, and
  infinite-dimensional {C}lifford Algebras}.
\newblock Ann. Phys.  \textbf{176} (1987), 49--113.

\bibitem {kraft:2018a}
\textsc{Kraft, A.:}\newblock \textit{BRST Reduction of
  Quantum Algebras with *-Involution}.
\newblock master thesis, Institute of Mathematics, University of
  {W}{\"u}rzburg, W{\"u}rzburg, Germany, 2018.

\bibitem {lance:1995a}
\textsc{Lance, E.~C.:}\newblock \textit{{H}ilbert
  {$C^*$}-modules. A Toolkit for Operator algebraists}, vol. 210 in
  \textit{London Mathematical Society Lecture Note Series}.
\newblock Cambridge University Press, Cambridge, 1995.

\bibitem {marsden.weinstein:1974a}
\textsc{Marsden, J., Weinstein, A.:}\newblock
  \textit{Reduction of symplectic manifolds with symmetry}.
\newblock Reports on Math. Phys.  \textbf{5} (1974), 121--130.

\bibitem {meinrenken:2013a}
\textsc{Meinrenken, E.:}\newblock \textit{Clifford
  algebras and {L}ie theory}, vol.~58 in \textit{Ergebnisse der
  Mathematik und ihrer Grenzgebiete. 3. Folge. A Series of Modern Surveys in
  Mathematics}.
\newblock Springer, Heidelberg, 2013.

\bibitem {morita:2001a}
\textsc{Morita, S.:}\newblock \textit{Geometry of
  differential forms}.
\newblock American Mathematical Society (AMS), Providence, RI, 2001.
\newblock Translation from the Japanese by Teruko Nagase and Katsumi Nomizu.

\bibitem {mueller-bahns.neumaier:2004a}
\textsc{M{\"u}ller-Bahns, M.~F., Neumaier, N.:}\newblock
  \textit{Some remarks on {$\mathfrak{g}$}-invariant Fedosov star
  products and quantum momentum mappings}.
\newblock J. Geom. Phys.  \textbf{50} (2004), 257--272.

\bibitem {neumaier:1998a}
\textsc{Neumaier, N.:}\newblock \textit{Sternprodukte
  auf {K}otangentenb{\"{u}}ndeln und {O}rdnungs-{V}orschriften}.
\newblock master thesis, Fakult{\"{a}}t f{\"{u}}r Physik,
  Albert-Ludwigs-Universit{\"{a}}t, Freiburg, 1998.

\bibitem {neumaier:2001a}
\textsc{Neumaier, N.:}\newblock
  \textit{Klassifikationsergebnisse in der
  {D}eformationsquantisierung}.
\newblock PhD thesis, Fakult{\"{a}}t f{\"{u}}r Physik,
  Albert-Ludwigs-Universit{\"{a}}t, Freiburg, 2001.
\newblock Available at {\url{https://www.freidok.uni-freiburg.de/data/2100}}.

\bibitem {neumaier:2002a}
\textsc{Neumaier, N.:}\newblock \textit{Local
  $\nu$-{E}uler Derivations and {D}eligne's Characteristic Class of {F}edosov
  Star Products and Star Products of Special Type}.
\newblock Commun. Math. Phys.  \textbf{230} (2002), 271--288.

\bibitem {reichert:2017a}
\textsc{Reichert, T.:}\newblock \textit{Characterstic
  classes of star products on Marsden-Weinstein reduced symplectic manifolds}.
\newblock Lett. Math. Phys.  \textbf{107} (2017), 643--658.

\bibitem {reichert:2017b}
\textsc{Reichert, T.:}\newblock \textit{Classification
  and Reduction of Equivariant Star Products on Symplectic Manifolds}.
\newblock PhD thesis, Institute of Mathematics, University of {W}{\"u}rzburg,
  W{\"u}rzburg, Germany, 2017.

\bibitem {reichert.waldmann:2016a}
\textsc{Reichert, T., Waldmann, S.:}\newblock
  \textit{Classification of Equivariant Star Products on Symplectic
  Manifolds}.
\newblock Lett. Math. Phys.  \textbf{106} (2016), 675--692.

\bibitem {tyutin:2008a}
\textsc{Tyutin, I.~V.:}\newblock \textit{Gauge
  Invariance in Field Theory and Statistical Physics in Operator Formalism}.
\newblock Preprint of P.N. Lebedev Physical Institute, No. 39, 1975,
  \textbf{arXiv:0812.0580} (2008).

\bibitem {waldmann:2007a}
\textsc{Waldmann, S.:}\newblock
  \textit{Poisson-{G}eometrie und {D}eformationsquantisierung.
  {E}ine {E}inf{\"u}hrung}.
\newblock Springer-Verlag, Heidelberg, Berlin, New York, 2007.

\bibitem {xu:1998a}
\textsc{Xu, P.:}\newblock \textit{Fedosov $*$-Products
  and Quantum Momentum Maps}.
\newblock Commun. Math. Phys.  \textbf{197} (1998), 167--197.

\end{thebibliography}
\end{document}
